\documentclass[]{article}

\usepackage{float}
\usepackage[utf8]{inputenc}

\usepackage{graphicx}
\usepackage[dvipsnames]{xcolor}
\usepackage{amsmath, amssymb, amsfonts, mathtools, todonotes, tikz}

\usepackage{amsmath, amssymb, amsfonts, amsthm}

\newtheorem{theorem}{Theorem}
\numberwithin{theorem}{section} 

\newtheorem*{theorem*}{Theorem}
\newtheorem{corollary}[theorem]{Corollary}
\newtheorem{lemma}[theorem]{Lemma}
\theoremstyle{definition}
\newtheorem{definition}[theorem]{Definition}
\newtheorem{remark}[theorem]{Remark}
\newtheorem{example}[theorem]{Example}

\DeclareMathOperator{\Dom}{Dom}
\DeclareMathOperator{\Aut}{Aut}
\DeclareMathOperator{\Stab}{Stab}

\DeclareMathOperator{\Sym}{Sym}
\DeclareMathOperator{\Alt}{Alt}

\newcommand{\ecol}{\ensuremath{\zeta}}
\newcommand{\vcol}{\ensuremath{\chi}}
\newcommand{\cthree}{\mathop{\overrightarrow{C_3}}}
\newcommand{\cfour}{\mathop{\overrightarrow{C_4}}}

\newcommand{\dunion}{\mathbin{\dot{\cup}}}

\usetikzlibrary{backgrounds}
\usetikzlibrary{positioning, calc}
\usetikzlibrary{fadings,decorations.pathmorphing,arrows.meta}
\usetikzlibrary{decorations.markings}
\tikzstyle{svertex}=[circle,inner sep=0.cm, minimum size=3mm, fill=black, draw=black]
\tikzstyle{xsvertex}=[circle,inner sep=0.cm, minimum size=2mm, fill=black, draw=black]
\tikzstyle{harc}=[thick, ->]
\tikzstyle{sharc} = [left color=black!5, right color=black,
	shading path={draw=transparent!20, ultra thick, ->, 
		decoration={pre length=0.25cm, post length=0.25cm}, 
	}]

\tikzset{shadearc/.style={
		postaction={
			decorate,
			decoration={
				markings,
				mark=at position \pgfdecoratedpathlength-0.05pt with {\arrow[Black,line width=#1] {>}; },
				mark=between positions 0 and \pgfdecoratedpathlength-2pt step 0.5pt with {
					\pgfmathsetmacro\myval{multiply(divide(
						\pgfkeysvalueof{/pgf/decoration/mark info/distance from start}, \pgfdecoratedpathlength),100)};
					\pgfsetfillcolor{Black!\myval!White};
					\pgfpathcircle{\pgfpointorigin}{#1};
					\pgfusepath{fill};}
}}}}

\tikzset{shadearc/.style={
		postaction={
			decorate,
			decoration={
				markings,
				mark=at position \pgfdecoratedpathlength-0.05pt with {\arrow[Black,line width=#1] {>}; },
				mark=between positions 0 and \pgfdecoratedpathlength-2pt step 0.5pt with {
					\pgfmathsetmacro\myval{multiply(divide(
						\pgfkeysvalueof{/pgf/decoration/mark info/distance from start}, \pgfdecoratedpathlength),100)};
					\pgfsetfillcolor{Black!\myval!Gray};
					\pgfpathcircle{\pgfpointorigin}{#1};
					\pgfusepath{fill};}
}}}}

\makeatletter
\newif\iftikz@shading@path

\title{Finite Vertex-colored Ultrahomogeneous\\ Oriented Graphs}
\author{Irene Heinrich, Eda Kaja, and Pascal Schweitzer}
\tikzset{>=stealth}
\begin{document}
	\maketitle
	
\begin{abstract}
	A relational structure $\mathcal{R}$ is ultrahomogeneous if every isomorphism of finite induced substructures of $\mathcal{R}$ extends to an automorphism of $\mathcal{R}$. We classify the ultrahomogeneous finite binary relational structures with one asymmetric binary relation and arbitrarily many unary relations. In other words, we classify the finite vertex-colored oriented ultrahomogeneous graphs. The classification comprises several general methods with which directed graphs can be combined or extended to create new ultrahomogeneous graphs. Together with explicitly given exceptions, we obtain exactly all  vertex-colored oriented ultrahomogeneous graphs this way. Our main technique is a technical tool that characterizes precisely under which conditions two binary relational structures with disjoint unary relations can be combined to form a larger ultrahomogeneous structure.
\end{abstract}
	
\section{Introduction}

A relational structure~$\mathcal{R}$ is \emph{ultrahomogeneous} if every isomorphism between finite induced substructures of $\mathcal{R}$ extends to an automorphism of $\mathcal{R}$.
Ultrahomogeneity is a natural generalization of transitivity: a graph is vertex-transitive (edge-transitive) if for each two vertices (edges) of the graph there exists an automorphism of the graph which maps the one vertex (edge) to the other.
In some sense, ultrahomogeneous structures form the most symmetric structures possible. Indeed, we can think of ultrahomogeneity as stating the following. If two parts of the structure locally look the same then there is a global symmetry of the object demonstrating that the parts are indeed structurally the same, even if the entire structure is taken into account.
The emergence and separate handling of highly symmetric structures is unavoidable in various algorithms for symmetry detection and exploitation. For example, Babai's celebrated quasipolynomial time algorithm~\cite{Babai16} for the graph isomorphism problem relies on the techniques of local certificates, which treats the case of complete symmetry separately. 

Being highly symmetric objects, ultrahomogeneous structures have been extensively studied over the years. In fact, there are numerous books, surveys, and major results on the matter (see for example~\cite{MR3418640,MR1434988,Fraisse1953,MR583847,Lachlan1997,DBLP:journals/dm/Macpherson11,Mekler1993,MR363974}). 
Beyond the intrinsic combinatorial interest in ultrahomogeneous structures, part of their appeal is their applicability in model theory in the form of stability theory,~$\omega$-categoricity, and Fra{\"{\i}}ss{\'{e}} limits~\cite{Ahlman2018LimitLaws,Fraisse1953}. They also have natural applications in the study of permutation groups and Ramsey theory~\cite{DBLP:journals/combinatorics/BottcherF13}.

Homogeneity is usually considered for countable structures. Having algorithmic applications in mind, however, in this paper we will focus exclusively on finite structures. In fact, even when only considering finite ultrahomogeneous structures, there exists an extensive body of research. 	
Finite simple graphs have been independently classified by Gardiner~\cite{MR419293} and by Gol'fand and Klin~\cite{GolfandKlin1978}. These graphs are, up to taking complements, disjoint unions of complete graphs all of the same order, the~$5$-cycle, and the line graph of the~$K_{3,3}$ (also known as~$3\times 3$ rook's graph).
Subsequently Lachlan classified  finite  ultrahomogeneous  digraphs~\cite{lachlan82-finite-homogeneous-simple-digraphs}. These include further infinite families and some exceptional graphs. The oriented graphs among them are described in  Theorem~\ref{thm: lachlan_asymmetric} and Figure~\ref{fig: uh-monochr-graphs} below.
Apart from graphs, finite ultrahomogeneous groups~\cite{https://doi.org/10.1112/S0024610700001484,https://doi.org/10.1112/jlms/s2-44.1.102,Li_1999} and finite ternary relational structures (sometimes called 3-graphs)  \cite{MR1373115} have been classified.

A major research program initiated by Cherlin that aimed at classifying ultrahomogeneous edge colored graphs (or equivalently structures with only binary relations) recently led to the classification of their automorphism groups (primitive binary permutation groups)~\cite{GillLiebeckSpiga2022}.

Crucially, none of the structures discussed so far have unary relations (in the terminology of graphs they are without loops or vertex colors). This implies in particular, that the structures are transitive. However, this limits the applicability of the results because no structures with different types of atoms (vertices) can be captured. Such types of vertices can be modeled with vertex colors.
For ultrahomogeneous vertex-colored graphs, research is limited. There is a classification when the color classes form independent sets~\cite{truss:countable-homogeneous-multipartite-graphs-2}. More generally, the ultrahomogeneous vertex-colored undirected finite graphs were recently classified~\cite{HeinrichSchneiderSchweitzer2020}.

\paragraph*{Results.} In this paper we classify ultrahomogeneous vertex-colored oriented finite graphs. This includes vertex-colored tournaments. In the language of relational structures this means we classify structures with a binary asymmetric relational structure and an arbitrary number of unary relations.

The classification essentially says the following. Up to certain forms of equivalence (namely color change and bichromatic symmetrization, see Definition~\ref{def:equivalence}) the graphs are color disjoint unions (i.e., disjoint unions of graphs with disjoint vertex-colors) of specific types of blow-ups of the following graphs:
\begin{enumerate}
	\item the graph $H_0$,
	\item a disjoint union of finitely many isomorphic copies of a discretely colored tournament,
	\item a graph in which each color class forms a directed triangle and each pair of color classes is joined by a directed 6-Cycle~$\overrightarrow{C_6}$.
	
\end{enumerate}

Figure~\ref{fig:complex:example:UHgraph} shows an example containing each of the three possible building blocks but without blow-ups. See Definition~\ref{definition:blow:up} for a formal definition of the specific blow-ups used in our classification. Also see Theorem~\ref{thm: last thm} for the formal theorem describing our classification. 	The monochromatic graphs that may appear as subgraphs induced by the color classes are described in Theorem~\ref{thm: lachlan_asymmetric} and shown in Figure~\ref{fig: uh-monochr-graphs}.

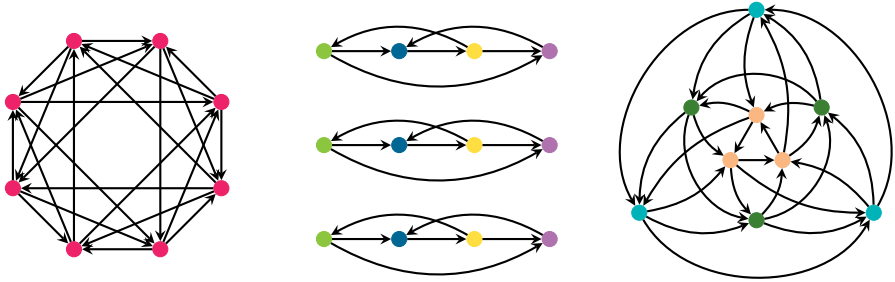
\begin{figure}
	\begin{tikzpicture}
		\begin{scope}
			\def\ynumber{8}
			\def\rad{1.5cm};
			\def\iangle{360/\ynumber}
			\foreach \i in {1,...,\ynumber}{
				\node[xsvertex, WildStrawberry] (x\i) at (270-\iangle/2+\i*\iangle:\rad) {};
			}
			\draw[harc] (x1) edge (x2) (x2) edge (x5) (x5)edge (x6) (x6) edge (x1);
			\draw[harc] (x3) edge (x4) (x4) edge (x7) (x7)edge (x8) (x8) edge (x3);
			\draw[harc] (x1) edge (x8) (x8) edge (x5) (x5)edge (x4) (x4) edge (x1);
			\draw[harc] (x3) edge (x2) (x2) edge (x7) (x7)edge (x6) (x6) edge (x3);
			\draw[harc] (x2) edge (x8) (x8) edge (x6) (x6)edge (x4) (x4) edge (x2);
			\draw[harc] (x7) edge (x1) (x1) edge (x3) (x3)edge (x5) (x5) edge (x7);
		\end{scope}
		
		\begin{scope}[shift={(3.75,0)}]
		  				\foreach \h in {-1.25,0,1.25}
		  				{
			  				\begin{scope}[shift={(0,\h)}]
				    					\node[xsvertex, LimeGreen] (r) at (-1,0) {};
				    					\node[xsvertex, MidnightBlue] (b) at (0,-0) {};
				    					\node[xsvertex, Goldenrod] (y) at (1,0) {};
				    					 \node[xsvertex, Orchid] (o) at (2,0) {};
				   					 \draw[harc] (r)edge(b) (b)edge(y) (y)edge(o);
				    					 \draw[harc] (y)edge[bend right=30](r) (r)edge[bend right=30](o) (o)edge[bend right=30](b);
			    					 
			  				\end{scope}[shift={(0,\h)}]
			    					
			    			}
		\end{scope}
		
		\begin{scope}[shift={(8.5,0)}]
		\def\irad{.4cm};
		\def\vxnumber{3}
		\def\angle{360/\vxnumber}
		\foreach \i in {1,...,\vxnumber}{
			\node[xsvertex, Apricot] (i\i) at (270-\angle/2+\i*\angle:\irad) {};
			}
		\draw[harc] (i1) edge (i2) (i2) edge (i3) (i3) edge (i1);
		\def\orad{1cm};
		\foreach \i in {1,...,\vxnumber}{
			\node[xsvertex, OliveGreen] (o\i) at (210-\angle/2+\i*\angle:\orad) {};
		}
		\draw[harc, bend right = 50] (o1) edge (o2) (o2) edge (o3) (o3) edge (o1);
		\draw[harc, bend right = 20] (i3) edge (o1) (o2) edge (i2);
		\draw[harc, bend right = 20] (i1) edge (o2) (o3) edge (i3);
		\draw[harc, bend right = 20] (i2) edge (o3) (o1) edge (i1);
		\def\yrad{1.8cm};
		\def\vxnumber{3}
		\def\angle{360/\vxnumber}
		\foreach \i in {1,...,\vxnumber}{
			\node[xsvertex, BlueGreen] (y\i) at (270-\angle/2+\i*\angle:\yrad) {};
		}
		\draw[harc, bend right = 60] (y1) edge (y2) (y2) edge (y3) (y3) edge (y1);
		\draw[harc, bend right = 20] (i3) edge (y1) (y2) edge (i2);
		\draw[harc, bend right = 20] (i1) edge (y2) (y3) edge (i3);
		\draw[harc, bend right = 20] (i2) edge (y3) (y1) edge (i1);
		
		\draw[harc, bend right = 25] (y3) edge (o1) (o2) edge (y2);
		\draw[harc, bend right = 25] (y1) edge (o2) (o3) edge (y3);
		\draw[harc, bend right = 25] (y2) edge (o3) (o1) edge (y1);
		\end{scope}

	\end{tikzpicture}
	\caption{Example of an ultrahomogeneous oriented vertex colored graph (no blow-ups).}
	\label{fig:complex:example:UHgraph}
\end{figure}

\paragraph*{Techniques.} 
We develop techniques to analyze under what conditions and how two monochromatic ultrahomogeneous graphs with different vertex colors can be non-trivially connected to form a new ultrahomogeneous graph.
The general extension theorem (Theorem~\ref{thm: general-extension}) describes five necessary and sufficient conditions for this. 
The crucial insight is that the automorphism group of the one side must be compatible with what we call an ultrahomogeneous system of partitions on the other side. 
This severely restricts pairs of ultrahomogeneous graphs that can be connected to create new ultrahomogeneous graphs. 

A second concept we introduce is a certain kind of a blow-up. Our concept here is more general than other blow-ups that have been previously used in the context of ultrahomogeneous structures. It allows for nontrivial connections between the 
replaced blocks. It thus allows us to organically recover some of the exceptional graphs as blow-ups of smaller graphs.

Overall these two techniques provide us with a clean and systematic way of analyzing highly symmetric graphs, dramatically reducing the number of cases that need to be considered.
In fact the classification follows using fairly easy counting arguments. 
We first classify bichromatic oriented ultrahomogeneous graphs (Theorem~\ref{bichromatic:theorem}).
We then extend the bichromatic case to the general case~(Theorem~\ref{thm: last thm}).
We highlight that our general extension theorem and the more general blow-ups are not particular to the oriented case and apply to vertex colored binary relational structures in general.

\section{Preliminaries}
For $n,n' \in \mathbb{N}$ we set $[n] \coloneqq \{1, 2, \dots, n\}$ and~$[n',n]\coloneqq \{n', \dots, n\}$.
For $i \in \mathbb{N}_{\geq 1}$ we denote the projection on the $i$-th coordinate by $\pi_i$, where the underlying set will
always be clear from context.
An \emph{ordered partition} $A$ of a set $V$ is a tuple $(P_1, P_2, \dots, P_k)$ of disjoint non-empty subsets of $V$ such that $\bigcup_{i \in [k]}P_i = V$. If $k = 1$, then $A$ is called the \emph{trivial} partition of $V$. A partition $(P_1, P_2, \dots, P_k)$ is \emph{discrete} if $|P_i| = 1$ for all $i \in [k]$. We use the symbol~$\dunion$ for the disjoint union.
For a set $V$ we and $\mathcal{S} \subseteq 2^V$ we set $\bigcup \mathcal{S} \coloneqq \bigcup_{S \in \mathcal{S}} S$.

\paragraph*{Digraphs.}
A \emph{directed graph} (short: \emph{digraph}) $G$ is a pair $(V, E)$ where $V$ is a non-empty set and $E \subseteq V^2$.
In this paper, all considered digraphs are \emph{finite} and \emph{loopless}.
An element of~$V$ is a \emph{vertex} of $G$ and an element of $E$ is an \emph{edge} of~$G$. With $V(G)$ and $E(G)$ we refer to the vertices and edges of~$G$, respectively.
A digraph~$G$ is an \emph{oriented graph} if $(u,v)\in E(D)$ implies $(v,u) \notin E(D)$ for all vertices $u$ and $v$ of~$G$.
If $E(G) = V^2\setminus \{(v,v) \colon v \in V(G)\}$, then $G$ is \emph{complete}.

\paragraph*{(Di-)graph families.}
Fix $n \in \mathbb{N}_{\geq 1}$.
The \emph{edgeless} (di-)graph $E_n$ has vertex set~$[n]$ and an empty edge set.
The \emph{directed cycle} $\overrightarrow{C_n}$ is the oriented graph on the vertex set $[n]$ with $E(\overrightarrow{C_n}) = \{(i,i+1)\colon i \in [n-1] \}\cup \{(n,1)\}$. For simplicity, we call~$\cthree$ the \emph{directed triangle}.

\paragraph*{Complete colored digraphs.}
A complete digraph $G$ with a vertex coloring~$\vcol_G$ and an edge coloring~$\ecol_G$ is a \emph{complete colored digraph (CCD)}.
Note that every directed graph $D$ with a vertex coloring $\vcol_D$ can be regarded as a CCD~$G$ with $V(G) = V(D)$, $\vcol_G = \vcol_D$, and $\ecol_G((u,v)) = 1$ if $(u,v) \in E(D)$ and $\ecol_G((u,v)) = 0$ otherwise.
Note that this translation from CCDs to vertex-colored directed graphs preserves (partial) isomorphisms.
Most of our techniques are formulated as statements on CCDs but in this manner they can easily be transferred to statements on directed graphs.
An inclusion-wise maximal subset $U$ of $V(G)$ with $|\vcol_G(U)| = 1$ is a \emph{vertex color class} of~$G$.
An \emph{edge color class} of $G$ is defined analogously.
If $G$ is a CCD and $U \subseteq V(G)$, then
the CCD~$G[U]$ with vertex set~$U$, vertex coloring $\chi_G|_{U}$, and edge coloring $\ecol_G|_{U \times U}$ 
is the \emph{induced subgraph of~$G$ by~$U$}.

\paragraph{Connectivity types.}
Let $R$ and $B$ be two disjoint vertex subsets of a CCD~$G$.
We say that $R$ and $B$ are \emph{homogeneously connected} if both ${\ecol_G}|_{R \times B}$ and ${\ecol_G}|_{B \times R}$ are constant. 
If $R$ and $B$ are not homogeneously connected and there exists a bijection $\alpha\colon R \to B$ where
${\ecol_G}|_{B \times R \setminus \{(\alpha(r),r)\colon r \in R\}}$ , ${\ecol_G}|_{\{(r, \alpha(r))\colon r \in R\}}$, ${\ecol_G}|_{\{(\alpha(r),r)\colon r \in R\}}$, and ${\ecol_G}|_{R \times B \setminus \{(r, \alpha(r))\colon r \in R\}}$ are constant, then $R$ and $B$ are \emph{matching-connected}.
If $|\vcol_G(V(G))| = |V(G)|$, then $\vcol_G$ is a \emph{discrete coloring.}

\paragraph{Isomorphisms and ultrahomogeneity.}
Let $G$ and $G'$ be two CCDs.
A bijection $\varphi\colon V(G) \to V(G')$ which satisfies for all vertices $u$ and $v$ in $V(G)$ that $\vcol_{G'}(\varphi(v)) = \vcol_G(v)$ and $\ecol_{G'}((\varphi(u),\varphi(v))) = \ecol_G((u,v))$ is an \emph{isomorphism}. If additionally $G = G'$, then $\varphi$ is an \emph{automorphism} of $G$.
The set of all automorphisms of $G$ forms a group under composition, which we denote by $\Aut(G)$.
An isomorphism of two induced subgraphs of $G$ is a \emph{partial isomorphism}.
A partial isomorphism $\varphi'\colon U \to W$ of $G$ \emph{extends to an automorphism} of $G$ if there exists~$\psi \in \Aut(G)$ such that $\psi|_{U} = \varphi'$.
A CCD $G$ is \emph{ultrahomogeneous} if every partial isomorphism of $G$ extends to an automorphism of $G$.

\paragraph{The wreath product of graphs.}
Let $D$ and $D'$ be two CCDs.
The \emph{wreath product}\footnote{Also called the \emph{lexicographic product.}} $D \cdot D'$  is a CCD with vertex set $V(D) \times V(D')$, with vertex colors
$\vcol_{D \cdot D'}(u,u') \coloneqq (\vcol_D(u), \vcol_{D'}(u'))$ for all $(u,u') \in V(D) \times V(D')$, and edge colors
\vspace{-0.25cm}
\begin{align*}
	\ecol_{D\cdot D'}((u,u'),(v,v'))\coloneqq
	\begin{cases}
		\ecol_{D'} ((u',v')) &\text{if $u = v$},\\
		\ecol_D ((u,v)) &\text{if $u \neq v$}.
	\end{cases}
\end{align*}

\paragraph*{Groups.}	
Let $V$ be a non-empty set.
We denote the \emph{symmetric group} of all permutations of $V$ by $\Sym(V)$.
A \emph{permutation group} $\Gamma$ on $V$ is a subgroup of~$\Sym(V)$.
For $v \in V$ and $\gamma \in \Gamma$ we set $v^{\gamma} \coloneqq \gamma(v)$.
An \emph{action} of $\Gamma$ on $V$ is a homomorphism $\phi$ from $\Gamma$ to $\Sym(V)$.
The image of an action of $\Gamma$ on $V$ is a subgroup of~$\Sym(V)$ called the permutation group \emph{induced} by $\Gamma$ on $V$, which we denote by $\Gamma^V$.
The \emph{orbit} of an element~$x$ in $V$ is the set $x^\Gamma\coloneqq \{x^\gamma\colon \gamma\in\Gamma\}$. 
We say that $\Gamma$ is \emph{transitive} on $V$ if $x^\Gamma=V$ for all $x\in V$. 
The \emph{stabilizer} of an element~$x$ in $V$ is the set $\Stab_\Gamma(x):=\{\gamma\in \Gamma \colon x^\gamma=x\}$. The \emph{setwise stabilizer of a subset $X$} of $V$ is the set $\Stab_\Gamma(X)$ of elements $\gamma \in \Gamma$ such that $X^\gamma=X$. The \emph{pointwise stabilizer of $X$} is the set $\mathrm{pw}\Stab(X)=\bigcap_{x\in X} \Stab_\Gamma(x)$.
A \emph{permutational isomorphism} between two permutation groups~$\Gamma\leq \Sym(V)$ and~$\Gamma'\leq \Sym(V')$ is a bijection~$\rho\colon V\rightarrow V'$ such that~$\Gamma' = \{ \rho \gamma \rho^{-1}\colon \gamma \in \Gamma \}$.

\paragraph{Block systems.} Let $\Gamma \leq \Sym(V)$ be transitive.
A \emph{block} is a subset $X$ of $V$ such that $X^\gamma = X$ or $X^\gamma\cap X=\emptyset$ for all $\gamma\in \Gamma$.
If $X=\{x\}$ for some $x\in V$ or $X=V$, then the block $X$ is \emph{trivial}, and otherwise it is \emph{non-trivial}.
If $X$ is a block, then the set $\{X^\gamma\colon \gamma\in \Gamma\}$ is an unordered partition of $V$ which is invariant under the action of $\Gamma$. In this case we call $\{X^\gamma\colon \gamma\in \Gamma\}$ a \emph{block system} of $V$.
A block system is \emph{trivial} if its blocks are trivial.
Note that every permutation in $\Gamma$ naturally induces a permutation of the blocks in a block system $\mathcal{B}$ of $V$. We denote the subgroup of $\Sym(\mathcal{B})$ which contains all permutations induced by permutations in $\Gamma$ by $\Gamma^{\mathcal{B}}$.

\paragraph{The wreath product of groups.}
Let $\Gamma\leq \Sym(V)$ and $\Gamma'\leq \Sym(W)$. The \emph{wreath product} of $\Gamma$ with $\Gamma'$, denoted $\Gamma\wr\Gamma'$ is the group of all permutations $\delta$ of $V\times W$ for which there exist $\gamma\in \Gamma$ and an element $\gamma_v'$ of $\Gamma'$ for each $v\in V$, such that 
$\delta((v,w))=(\gamma(v),\gamma_v'(w)) \text{ for every } (v,w)\in V\times W$.

\begin{remark}
	If $G$ and $G'$ are two edge-color disjoint CCDs with respective automorphism groups $\Gamma$ and $\Gamma'$, then the automorphism group of $G \cdot G'$ is permutationally isomorphic to $\Gamma'\wr\Gamma$.
\end{remark}

\paragraph{Families of groups.}
Fix $n \in \mathbb{N}_{\geq 1}$.
Set $\Sym(n)\coloneqq \Sym([n])$.
We denote the cyclic group of order $n$ by $\mathbb{Z}_n$ and the alternating group on $n$ vertices by $\Alt(n)$.

\paragraph{Monochromatic graphs.} Lachlan's~\cite{lachlan82-finite-homogeneous-simple-digraphs} classic results classify the mono\-chromatic ultrahomogeneous graphs. The oriented graphs among them, which are relevant to our classification, are depicted in Figure~\ref{fig: uh-monochr-graphs} and are as follows. \begin{theorem}[\cite{lachlan82-finite-homogeneous-simple-digraphs}]\label{thm: lachlan_asymmetric}
	An oriented graph is ultrahomogeneous if and only if it is isomorphic to one of
	$\cfour$, $E_n$, $E_n\cdot \cthree$, $\cthree\cdot E_n$, or $H_0$ for some $n\ \in \mathbb{N}\setminus \{0\}$.
\end{theorem}

\begin{figure}[h!]
	\centering
	\begin{tikzpicture}[scale=.9]
		\begin{scope}[shift={(-.5,0)}]
			\def\ynumber{4}
			\def\rad{1cm};
			\def\iangle{360/\ynumber}
			\foreach \i in {1,...,\ynumber}{
				\node[xsvertex,WildStrawberry] (x\i) at (270-\iangle/2+\i*\iangle:\rad) {};
			}
		
		\begin{scope}[on background layer]
			\draw[shadearc=.6] (x1) -- (x2);
			\draw[shadearc=.6] (x2) -- (x3);
			\draw[shadearc=.6] (x3) -- (x4);
			\draw[shadearc=.6] (x4) -- (x1);
		\end{scope}
		\end{scope}
		\node[] (c4) at (-.5,-2) {$\cfour$};
		\begin{scope}[shift={(4,0)}]
			\def\ynumber{8}
			\def\rad{1.5cm};
			\def\iangle{360/\ynumber}
			\foreach \i in {1,...,\ynumber}{
				\node[xsvertex, WildStrawberry] (x\i) at (270-\iangle/2+\i*\iangle:\rad) {};
			}
			\begin{scope}[on background layer]
				\draw[shadearc=.6] (x1) --(x2);
				\draw[shadearc=.6] (x2) -- (x5); 
				\draw[shadearc=.6] (x5) -- (x6);
				\draw[shadearc=.6] (x6) -- (x1);
				\draw[shadearc=.6] (x3) -- (x4);
				\draw[shadearc=.6] (x4) -- (x7);
				\draw[shadearc=.6] (x7) -- (x8);
				\draw[shadearc=.6] (x8) -- (x3);
				\draw[shadearc=.6] (x1) -- (x8);
				\draw[shadearc=.6] (x8) -- (x5);
				\draw[shadearc=.6] (x5) -- (x4);
				\draw[shadearc=.6] (x4) -- (x1);
				\draw[shadearc=.6] (x3) -- (x2);
				\draw[shadearc=.6] (x2) -- (x7);
				\draw[shadearc=.6] (x7) -- (x6);
				\draw[shadearc=.6] (x6) -- (x3);
				\draw[shadearc=.6] (x2) -- (x8);
				\draw[shadearc=.6] (x8) -- (x6);
				\draw[shadearc=.6] (x6) -- (x4);
				\draw[shadearc=.6] (x4) -- (x2);
				\draw[shadearc=.6] (x7) -- (x1);
				\draw[shadearc=.6] (x1) -- (x3);
				\draw[shadearc=.6] (x3) -- (x5);
				\draw[shadearc=.6] (x5) -- (x7); 
			\end{scope}

		\end{scope}
		\node[] (h0) at (4,-2) {$H_0$};
		\begin{scope}[shift={(8,0)}]
			\def\ynumber{7}
			\def\mrad{1cm};
			\def\iangle{360/\ynumber}
			\foreach \i in {1,...,\ynumber}{
				\node[xsvertex, WildStrawberry] (x\i) at (210-\iangle/2+\i*\iangle:\mrad) {};
			}	
		\end{scope}
		\node[] (h0) at (8,-2) {$E_n$};
		\begin{scope}[shift={(1,-4)}, scale=.6]
			\def\vxnumber{3}
			\def\angle{360/\vxnumber}
			\def\orad{1.5cm};
			\foreach \i in {1,...,\vxnumber}{
				\node[] (o\i) at (210-\angle/2+\i*\angle:\orad) {
					\begin{tikzpicture}
						\def\irad{.4cm};
						\def\vxnumber{3}
						\def\angle{360/\vxnumber}
						\foreach \i in {1,...,\vxnumber}{
							\node[xsvertex, WildStrawberry] (i\i) at (270-\angle/2+\i*\angle:\irad) {};
						}
					\begin{scope}[on background layer]
						\draw[shadearc=.6] (i1)--(i2);
						\draw[shadearc=.6] (i2)--(i3);
						\draw[shadearc=.6] (i3)--(i1);
					\end{scope}
				\end{tikzpicture} };
			}
			\node[] (enc3) at (0,-3) {$E_n\cdot \cthree$};
			\node[] (last) at (8,-3) {$\cthree\cdot E_n$};
		\end{scope}
		\begin{scope}[shift={(7,-4)}, scale=.8]
			\def\irad{.4cm};
			\def\vxnumber{3}
			\def\angle{360/\vxnumber}
			\def\orad{1.5cm};
			\foreach \i in {1,...,\vxnumber}{
				\node[circle, inner sep=-1.5] (o\i) at (210-\angle/2+\i*\angle:\orad) {
					\begin{tikzpicture}
						\def\ynumber{5}
						\node[circle, minimum size = 1.2cm, fill = gray, opacity=.5] (background) at (0,0) {};
						\def\mrad{.4cm};
						\def\iangle{360/\ynumber}
						\foreach \i in {1,...,\ynumber}{
							\node[xsvertex, WildStrawberry] (x\i) at (210-\iangle/2+\i*\iangle:\mrad) {};
						}
				\end{tikzpicture} };
			}
			\draw[harc, bend right = 40] (o1) edge (o2);
			\draw[harc, bend right = 40] (o2) edge (o3);
			\draw[harc, bend right = 40] (o3) edge (o1);
		\end{scope}
		
	\end{tikzpicture}
	
	\caption{The ultrahomogeneous oriented (monochromatic) graphs. For each of the three infinite families
		$\{E_n\colon n \in \mathbb{N}_{\geq 1} \}$, 
		$\{E_n \cdot \protect\cthree\colon n \in \mathbb{N}_{\geq 1} \}$, and $\{\protect\cthree \cdot E_n\colon n \in \mathbb{N}_{\geq 1}\}$ we show one example.
		An arrow from one gray circle to another indicates all arcs from vertices in the first circle to vertices in the second are present.}
	\label{fig: uh-monochr-graphs}
\end{figure}
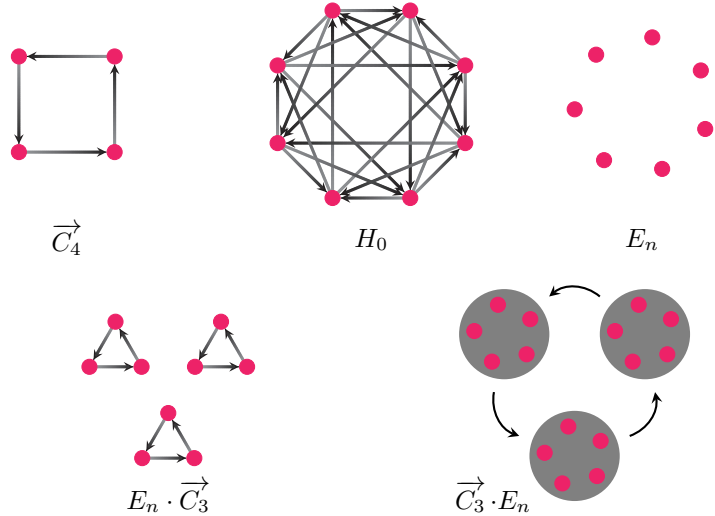

	\section{Extension theorems} \label{section: extension theorem}	
	Let $G$ be a CCD, $A$ an ordered partition of $V(G)$, and $\varphi \in \Aut(G)$.
	By~$\varphi(A)$ we denote the ordered partition of $V(G)$ with $\pi_i(\varphi(A)) = \varphi(\pi_i(A))$ for all $i \in [|A|]$.
	Let $A_1, A_2, \dots, A_k$ be a finite sequence of partitions of $V(G)$.
	We set
	\[\vcol_G(A_1, A_2, \dots, A_k)\colon V \to \vcol_G(V(G)) \times \mathbb{N}^k \quad v \mapsto (\vcol_G(v),i_1,\ldots,i_k), \]
	where~$\forall j\in [k]\  v \in \pi_{i_j}(A_{j})$.
	That is, two vertices $u$ and $v$ of $G$ have the same color with respect to the coloring $\vcol_G(A_1, A_2, \dots, A_k)$ precisely if they have the same color with respect to $\vcol_G$ and for each $j \in [k]$ the two vertices $u$ and $v$ lie in the same part of the partition $A_j$.

	\begin{definition}[Ultrahomogeneous system of partitions]
		Let $G$ be a CCD.
		An ordered partition $A$ of $V(G)$ is \emph{ultrahomogeneous} if~$G$ with the vertex coloring $\vcol_G(A)$ is ultrahomogeneous.
		We set $\mathcal{A}(A) \coloneqq \{\varphi(A) \colon \varphi \in \Aut(G)\}$.
		If~$A$ is an ordered partition of~$V(G)$ such that
		for every finite sequence $A_1, A_2, \dots, A_{\ell}$ of partitions in $\mathcal{A}(A)$ the CCD~$G$ with the vertex coloring $\vcol_G(A_1, A_2, \dots, A_{\ell})$ is ultrahomogeneous, then we call $\mathcal{A}(A)$ an \emph{ultrahomogeneous system of partitions}.
		If $|\mathcal{A}(A)| \neq 1$, then $\mathcal{A}(A)$ is called \emph{non-trivial}.
	\end{definition}
	
	\begin{definition}[Easygoing CCD with respect to a block system]
		Let~$H$ be a vertex-monochromatic CCD with a block system $\mathcal{B}$. We call~$H$ \emph{easygoing} with respect to~$\mathcal{B}$ if each subset~$\mathcal{B}'\subseteq \mathcal{B}$ satisfies
		\[ \left(\mathrm{pw}\Stab_{\Aut(H)} \left(\bigcup \mathcal{B}'\right)\right)^\mathcal{B}=  \Stab_{\Aut(H)^\mathcal{B}}\left(\mathcal{B}'\right),\]
		i.e., the pointwise stabilizer of the set~$\bigcup \mathcal{B}' \coloneqq \bigcup_{B \in \mathcal{B}'} B$ 
		in the group~$\Aut(H)$ induces the same group on~$\mathcal{B}$ as the pointwise stabilizer of~$\mathcal{B}'$ in the induced group~$\Aut(H)^\mathcal{B}$.
	\end{definition}
	
	Note that every block system $\mathcal{B}$ of a vertex-monochromatic CCD $H$ satisfies $\left(\mathrm{pw}\Stab_{\Aut(H)} \left(\bigcup \mathcal{B}'\right)\right)^\mathcal{B} \subseteq \Stab_{\Aut(H)^\mathcal{B}}\left(\mathcal{B}'\right).$
	The property of being easygoing will be crucial for us to be able to specify how the vertices inside~$\bigcup \mathcal{B}'$ are mapped while still being able to choose how blocks outside~$\mathcal{B}'$ are mapped. This is done as follows.
	\begin{lemma}\label{lem:application:of:easygoing}
		Let $H$ be a vertex-monochromatic CCD which is easygoing with respect to a block system $\mathcal{B}$.
		If two automorphisms~$\psi$ and $\varphi$ in $\Aut(H)$ satisfy  $\psi^{\mathcal{B}}|_{\mathcal{B}'}=\varphi^{\mathcal{B}}|_{\mathcal{B}'}$,
		then there is~$\tau \in \Aut(H)$ such that~$\tau^{\mathcal{B}}=\psi^{\mathcal{B}}$ and~$\tau (x)=\varphi(x)$ for all~$x\in\bigcup \mathcal{B}'$.
	\end{lemma}
	\begin{proof}
		If~$\psi$ and~$\varphi$ are as given, then~$(\varphi^{-1}\circ \psi)^{\mathcal{B}} \in \Stab_{\Aut(H)^\mathcal{B}}(\mathcal{B}') $.
		Since~$H$ is easygoing there is~$\rho \in \Aut(H)$ which is the identity on~$\bigcup \mathcal{B}'$ and satisfies~$\rho^{\mathcal{B}}= (\varphi^{-1}\circ \psi)^{\mathcal{B}}$. The map~$\tau\coloneqq \varphi\circ \rho$ is an automorphism of $H$ with the desired properties.
	\end{proof}
	
	\begin{theorem}[General extension theorem] \label{thm: general-extension}
		Let $G$ be a CCD on precisely two vertex color classes $R$ and $B$ such that $\ecol_G(B\times R) =[\ell]$. 
		The graph $G$ is ultrahomogeneous if and only if all of the following conditions are satisfied:
		\begin{enumerate}
			\item \label{itm: induced UH} Both graphs $G[R]$ and $G[B]$ are ultrahomogeneous.
			\item \label{itm: uh refined nbhds} 
			For $b \in B$ let $A(b)\coloneqq (P_1,\ldots,P_\ell)$ where~$P_i= \{r \in R \colon \ecol_G((b,r))=i \}$.
			The set $\mathcal{A} \coloneqq \{A(b)\colon b \in B\}$ is an ultrahomogeneous system of partitions of $G[R]$.
			\item \label{itm: block-system-B} For $A \in \mathcal{A}$ set $X(A) \coloneqq \{b \in B\colon A(b)=A\}$.
			The set $\mathcal{B}\coloneqq \{X(A)\colon A \in \mathcal{A} \}$ is a block system of $\Aut(G[B])$.
			\item \label{itm: block-group} $\Aut(G[B])^{\mathcal{B}}=\{\widehat{\varphi}\colon \varphi \in \Aut(G[R])^{\mathcal{A}} \}$ where $\widehat{\varphi}\colon \mathcal{B} \to \mathcal{B}$, $X(A) \mapsto X(\varphi(A))$ for all $A \in \mathcal{A}$.
			\item \label{itm: easygoing} $G[B]$ is easygoing with respect to~$\mathcal{B}$.
		\end{enumerate}
	\end{theorem}
	
	\begin{proof}
		($\Rightarrow$) Assume that $G$ is ultrahomogeneous.
		
		\medskip
		(\textit{Part~\ref{itm: induced UH}}): 
		Let $\varphi:U\to W$ be a partial isomorphism of $G[R]$.
		Since $G$ is ultrahomogeneous, there exists $\psi \in \Aut(G)$ such that $\psi$ extends $\varphi$. 
		By definition, every automorphism of $G$ preserves vertex color classes. 
		In particular, $\psi(R)=R$ and, hence,  $\psi|_{R}$ is an automorphism of $G[R]$ which extends $\varphi$. 
		Altogether $G[R]$ is ultrahomogeneous.
		Replacing the roles of $R$ and $B$ in this proof yields the analogous statement for $G[B]$.
		
		\medskip
		(\textit{Part~\ref{itm: uh refined nbhds}}):
		Let $b_1, \ldots, b_k \in B$ and let $\varphi$ be a partial isomorphism of $G[R]$ with the vertex coloring $\vcol_G(A(b_1), \ldots, A(b_k))$.
		Since $\varphi$ respects the coloring $\vcol_G(A(b_1), \ldots, A(b_k))$ we may extend $\varphi$ to a partial isomorphism $\varphi'$ on the domain $\Dom(\varphi) \cup \{b_1, b_2, \ldots, b_k\}$ such that $\varphi'$ is the identity on $\{b_1, b_2, \ldots, b_k\}$.
		Since~$G$ is ultrahomogeneous there exists $\psi$ in $\Aut(G)$ which extends $\varphi'$. 
		The map $\psi|_R$ is an automorphism of $G[R]$ which extends $\varphi$ and respects the coloring $\vcol_G(A(b_1), \ldots, A(b_k))$.
		Altogether $\mathcal{A}$ is an ultrahomogeneous system of partitions of $G[R]$.	
		
		\medskip 
		(\textit{Part~\ref{itm: block-system-B}}):
		Two vertices~$b$ and $b'$ in $B$ are in the same part of~$\mathcal{B}$ precisely if~$A(b)=A(b')$.
		If they are in the same part, then for every automorphism  $\varphi \in \Aut(G)$ we have that~$A(\varphi(b))=A(\varphi(b'))$, so their images are in the same part of~$\mathcal{B}$. Thus~$\mathcal{B}$ is a block system.
		
		\medskip 
		(\textit{Part~\ref{itm: block-group}}):
		Since automorphisms preserve edge colors every $\psi \in \Aut(G)$ satisfies that
		\[(\psi|_B)^{\mathcal{B}} = \widehat{(\psi|_R)^{\mathcal{A}} }.\]
		The statement follows since every automorphism of $G[R]$ (of $G[B]$) extends to an automorphism of $G$ since $G$ is ultrahomogeneous.

		\medskip
		(\textit{Part~\ref{itm: easygoing}}): 
		Fix~$\mathcal{B}'\subseteq \mathcal{B}$.
		We have
		$(\mathrm{pw}\Stab_
		{\Gamma} (\bigcup \mathcal{B}'))^\mathcal{B}\subseteq  \Stab_{\Gamma^\mathcal{B}}(\mathcal{B}')$ for arbitrary groups $\Gamma$.
		So we need to show the other inclusion for the group~$\Gamma\coloneqq \Aut(G[B])$.
		If~$\psi \in \Stab_{\Gamma^\mathcal{B}}(\mathcal{B}')$, then by Property~\ref{itm: block-group} there is a permutation~$\varphi \in \Aut(G[R])^{\mathcal{A}}$ with~$\widehat{\varphi}=\psi$. We extend~$\varphi$ to a partial isomorphism $\varphi'$ on~$R\cup \bigcup \mathcal{B}'$ which fixes all vertices in~$\bigcup \mathcal{B}'$. This is possible since~$\varphi$ leaves the ordered partitions~$A(b)$ with~$b\in \bigcup \mathcal{B}'$ invariant. Since~$G$ is ultrahomogeneous~$\varphi'$ extends to an automorphism of~$G$ that stabilizes the sets in~$\mathcal{B}'$ and induces~$\psi$ on~$\mathcal{B}$.

		\medskip
		($\Leftarrow$) Now assume that the Conditions~\eqref{itm: induced UH}--\eqref{itm: easygoing} are satisfied.
		Let $\varphi\colon U \to W$ be a partial isomorphism of $G$.
		
		If $U \cap B \neq \emptyset$, then by~\eqref{itm: induced UH} there exists an automorphism $\psi_B \in \Aut(G[B])$ which extends $\varphi|_{B}$.
		Condition~\eqref{itm: block-group} implies that there is an automorphism $\psi_R \in \Aut(G[R])$ such that $\widehat{\psi_R^{\mathcal{A}}} = \psi_B^{\mathcal{B}}$.
		Fix an ordering $b_1, b_2, \dots, b_{|B|}$ of $B$.
		Since $\varphi$ is a partial isomorphism of $G$ we have for all $r \in U \cap R$ that
		$$\vcol(A(b_1), A(b_2),\dots, A(b_{|B|}))(\varphi(r)) = \vcol(A(b_1),A(b_2), \dots, A(b_k))(\psi_R(r)).$$
		By Condition~\eqref{itm: uh refined nbhds} there exists an automorphism $\rho \in \Aut(G[R])$ which respects the coloring $\vcol(A(\varphi(b_1)), \dots, A(\varphi(b_{|B|})))$ and satisfies $\rho(\psi_R((r)) = \varphi(r)$ for all $r \in U\cap R$.
		Altogether the map $\psi$ with $\psi|_{R} = \rho \circ \psi_R$ and $\psi|_{B} = \psi_B$ is an automorphism of $G$ which extends $\varphi$ on~$A$ and induces~$\psi_B^{\mathcal{B}}$ on~$\mathcal{B}$. Since~$G[B]$ is easygoing (Condition~\eqref{itm: easygoing}) there is an automorphism~$\tau$ that is the identity on~$R$ and~$\psi_B\psi^{-1}$ on~$B$. Then~$\tau \psi$ is an automorphism that extends~$\varphi$.
		
		If $U \subseteq R$, then by Condition~\eqref{itm: induced UH} there exists $\rho \in \Aut(G[R])$ which extends~$\varphi$. According to Condition~\eqref{itm: block-group} there is an automorphism of $G$ which extends $\rho$.
	\end{proof}

	We now investigate the case when the partition~$A(b)$ is a block system of~$R$ for some (and thus for every)~$b\in B$. For this case we are interested in minimal extensions, that is, extensions in which the sets~$X(A)$ from the theorem contain only a single element~$b\in B$.

	\begin{theorem}[Minimal extension theorem]\label{thm:min-ext-thm}
		Let $G$ be a CCD on precisely one vertex color class $R$ and let~$A^{\star}$ be an ordered partition such that~$\mathcal{A}:=\mathcal{A}(A^{\star})$
		is an ultrahomogeneous system of partitions. 
		If, as an unordered partition,~$A^{\star}$ is a block system of~$\Aut(G[R])$,
		then the CCD~$\widehat{G}$ on vertex set~$ \{b_A \colon A \in \mathcal{A}\}\cup R$ with~$\widehat{G}[R]=G$ and with
		\begin{align*}
			\vcol_{\widehat{G}}(v) &= \text{red for all $v \in V(G)$},\\
			\vcol_{\widehat{G}}(b_A) &= \text{blue for all $A \in \mathcal{A}$},  \\
			\ecol_G((r,b_A)) &= i~\text{if $r \in \pi_i(A)$ for all $r \in R$ and $A \in \mathcal{A}$, and}\\
			\ecol_G((b_A, b_{A'})) &= \text{Iso-Type}((G,\vcol(A,A'))),
		\end{align*}
		where $\text{Iso-Type}((G,\vcol(A,A')))$ is the class of colored graphs that are isomorphic to~$(G,\vcol(A,A'))$, is ultrahomogeneous.
		In this case we call $G$ the  \emph{minimal ultrahomogeneous extension of $G[R]$ with respect to $\mathcal{A}$.}
	\end{theorem}
	
	\begin{proof}
		Set $B\coloneqq \{b_A \colon A \in \mathcal{A}\}$.
		We use the general extension theorem and show that its five conditions are satisfied. Condition~2 is satisfied since~$A(b_A)=A$ by construction.
		Conditions~3  and~\ref{itm: easygoing} are satisfied since~$\mathcal{B}$ is a discrete partition.
		
		The interesting condition is Condition~1.
		Suppose~$\varphi:U\to W$ is a partial isomorphism of~$G[B]$. Let $t$ be the number of parts of~$A$.
		Since $A^{\star}$ is a block system for each two ordered partitions $A$ and $A'$ in $\mathcal{A}$ there exists $\tau_{A,A'}\in \Sym([t])$  with~$\tau_{A,A'}(i)=j$ if~$\pi_{i}(A)= \pi_j(A')$.

		Note that for four partitions~$A,A',\overline{A},\overline{A'}\in \mathcal{A}$ the graphs~$(G,\vcol(A,A'))$ and $(G,\vcol(\overline{A},\overline{A'}))$ are isomorphic exactly if~$\tau_{A,A'}=\tau_{\overline{A},\overline{A'}}$. This is the case exactly if the edges~$(b_A,b_{A'})$ and~$(b_{\overline{A}},b_{\overline{A'}})$ have the same color.
		Moreover, for all triples~$A$, $A'$, and $\overline{A}$ there is exactly one~$\overline{A'}$ so that~$\tau_{A,A'}= \tau_{\overline{A},\overline{A'}}$. 
		Choose~$b_A\in U$ and suppose~$\varphi(b_A)= b_{\overline{A}}$.
		
		We let~$\widehat{\varphi}$ be the map that sends~$b_{A'}$ to the vertex~$b_{\overline{A'}}$ with~$\tau_{A,A'}=\tau_{\overline{A},\overline{A'}}$.
		This map is an isomorphism of~$G[B]$. It is an extension of~$\varphi$ since the choice of~$\overline{A'}$ is unique. 
		
		Regarding Condition~4 
		note that by our previous observations in the construction~$G[B]$ is the Cayley graph of~$\Aut(G[R])^\mathcal{A}$ with respect to the generating set that contains all non-trivial elements.
		This means in particular that the group $\Aut(G[B])^{\mathcal{B}}$ is regular. Thus the order of $\Aut(G[B])^{\mathcal{B}}$ is no larger than the order of~$\Aut(G[R])^{\mathcal{A}}$. Conversely, each permutation of the parts of~$A$ induces a corresponding permutation of~$B$.
	\end{proof}

	\begin{example}
		Let~$G\cong\cfour$ and label its vertices $r_1$, $r_2$, $r_3$ and $r_4$ in a cyclic fashion.
		Let $A=(P_1,P_2):=(\{r_1,r_3\},\{r_2,r_4\})$.
		Then $\mathcal{A}(A)=\{A,A'\}$ with $A':=(P_2,P_1)$ is an ultrahomogeneous system of partitions. 
		Note that $A$ is a block system for $\Aut(G)$. 
		The CCD $\widehat{G}$ with $\widehat{G}[R]=G$ as in Theorem~\ref{thm:min-ext-thm} has the following properties. 
		First, $B=\{b_A,b_{A'}\}$.
		Second, for $i\in\{1,2\}$, $\ecol((r_i,b_A))=\ecol((r_{i+2},b_A))=i$ since $r_i,r_{i+2}\in P_i$. 
		It follows that $\ecol((r_i,b_{A'}))=1$ if $\ecol((r_i,b_A))=2$ and $\ecol((r_i,b_{A'}))=2$ if $\ecol((r_i,b_A))=1$ for $i\in [4]$.
		Lastly, $(G,\vcol(A,A'))$ assigns the color $(\text{red},1,2)$ to $r_1$ and $r_3$ and the color $(\text{red},2,1)$ to $r_2$ and $r_4$.
		So $\ecol(b_A,b_{A'})=\ecol(b_{A'}, b_A)$.
		Altogether we obtain that $\widehat{G}$ is indeed the minimal ultrahomogeneous extension of a red $\cfour$ with respect to $\mathcal{A}(A)$ (see Figure~\ref{fig: R_c4}, left). 
	\end{example}
	
	\section{Graph operations which preserve ultrahomogeneity}
	\label{sec: preserve uh}
	
	Two CCDs $G$ and $H$ are \emph{vertex-color disjoint} if $\vcol_G(V(G)) \cap \vcol_H(V(H)) = \emptyset$.
	
	\begin{definition}[Color disjoint union] 
		Let $G$ and $H$ be two vertex-color disjoint CCDs and $c$ a new edge color, i.e., $c \notin \ecol_G(E(G)) \cup \ecol_H(E(H))$.
		The \emph{color disjoint union} of $G$ and $H$, denoted~$G \mathbin{\square} G$, is a CCD with vertex set $V(G) \dot{\cup} V(H)$ such that the vertices in $G\mathbin{\square} H$ inherit their original vertex colors, and edge colors
		\begin{align*}
			\ecol_{G\mathbin{\square}H}((u,v))\coloneqq
			\begin{cases}
				\ecol_G ((u,v)) &\text{if $u,v\in G$},\\
				\ecol_{H} ((u,v)) &\text{if $u,v\in H$},\\
				c &\text{otherwise}. 
			\end{cases}
		\end{align*}
	\end{definition}
	\begin{lemma}\label{lem: color-disjoint union} The color disjoint union of two CCDs $G$ and $H$ is ultrahomogeneous if and only if $G$ and $H$ are ultrahomogeneous.
	\end{lemma}
	\begin{proof}
		Isomorphisms preserve colors and the CCDs $G$ and $H$ are vertex-color disjoint.
		Hence, every partial isomorphism of $G\mathbin{\square} H$ induces a partial isomorphism of $G$ and a partial isomorphism of $H$.
		 Conversely, two partial isomorphisms, one of $G$ and one of $H$, can be combined to a partial isomorphism of $G\mathbin{\square} H$ since $G$ and $H$ are vertex-color disjoint and $V(G)$ and $V(H)$ are homogeneously connected in $G\mathbin{\square} H$.
	\end{proof}

	If $H$ is a CCD and $f\colon \ecol_H(E(H)) \to S$ is a bijection, then the CCD $G$ with $V(G) = V(H)$, $\vcol_G = \vcol_H$, and~$\ecol_G = f(\ecol_H)$ is an \emph{edge color change} of $H$.
	A \emph{vertex color change} is defined analogously using~$f(\chi_H)$ for some bijection $f\colon \chi_H(V(H)) \to S$.
	
	A CCD $G$ is a \emph{bichromatic symmetrization} of a CCD~$H$
	if
	$V(G) = V(H)$, $\vcol_G = \vcol_H$, and there exist
	two distinct edge colors~$c$ and $d$ in $\ecol_H(E(H))$ and two distinct vertex colors~$r$ and $b$ in $\chi_H(V(H))$ such that
		\[\ecol_G((u,v))= 
	\begin{cases}
		\ecol_H((v,u)) & \text{if } \ecol_H((v,u))=c, \ecol_G((u,v)) = d,\\& \phantom{if } \chi_H(u) = b, \text{ and }  \chi_H(v)=r ,\\
		\ecol_H ((u,v)) &\text{otherwise}.
	\end{cases}\]
	If $G$ is a bichromatic symmetrization of $H$, then $H$ is called an \emph{inverse bichromatic symmetrization} of $G$.

	\begin{definition}[Equivalence up to color changes and bichromatic symmetrization]\label{def:equivalence} 
		We call two CCDs $G$ and~$H$ \emph{equivalent up to color changes
			and bichromatic symmetrization}, denoted~$G\sim_c H$, if~$G$ can be obtained from~$H$ by a sequence of vertex color changes,
		edge color changes,
		bichromatic symmetrizations, or inverse bichromatic symmetrizations.
	\end{definition}
	
	We also use terms such as \emph{equivalent up to vertex color changes}, \emph{equivalent up to edge color changes}, and
	\emph{equivalent up to (inverse) bichromatic symmetrization}.
	\begin{example}[Bichromatic symmetrization of an oriented graph]
		Consider the left graph of Figure~\ref{fig: symmetrizations}.
		In order to apply a sequence of bichromatic symmetrizations and inverse bichromatic symmetrizations, we translate this graph into a CCD by inserting all possible directed edges and labelling added edges with a new color (orange, see second graph from the left).
		Now, we apply bichromatic symmetrization to obtain the graph in the middle.
		An inverse bichromatic symmetrization yields the second graph from the right.
		A translation back to the setting of oriented graphs finally results into the graph on the right.
	\end{example}
	
	\begin{figure}
		\centering
		\begin{tikzpicture}[scale=.6]
			\begin{scope}[shift={(0,0)}]
				\node[xsvertex, WildStrawberry] (r1) at (-1,-0.75) {};
				\node[xsvertex, WildStrawberry] (r2) at (-1,0.75) {};
				\node[xsvertex, MidnightBlue] (b1) at (1,-0.75) {};
				\node[xsvertex, MidnightBlue] (b2) at (1,0.75) {};
				\draw[harc] (r1)edge(b1) (r2)edge(b2);
			\end{scope}
			
			\begin{scope}[shift={(3.7,0)}]
				\node[xsvertex, WildStrawberry] (r1) at (-1,-0.75) {};
				\node[xsvertex, WildStrawberry] (r2) at (-1,0.75) {};
				\node[xsvertex, MidnightBlue] (b1) at (1,-0.75) {};
				\node[xsvertex, MidnightBlue] (b2) at (1,0.75) {};
				\draw[harc, bend left=10] (r1)edge(b1) (r2)edge(b2);
				\draw[harc, bend left=10, Melon] (b1)edge(r1) (b2)edge(r2);
				\draw[harc, bend left=10, Melon] (r1)edge(r2) (b1)edge(b2) (r2)edge(r1) (b2)edge(b1);
				\draw[harc, bend left=10, Melon] (r1)edge(b2) (r2)edge(b1) (b1)edge(r2)  (b2)edge(r1);
			\end{scope}
			
			\begin{scope}[shift={(7.4,0)}]
				\node[xsvertex, WildStrawberry] (r1) at (-1,-0.75) {};
				\node[xsvertex, WildStrawberry] (r2) at (-1,0.75) {};
				\node[xsvertex, MidnightBlue] (b1) at (1,-0.75) {};
				\node[xsvertex, MidnightBlue] (b2) at (1,0.75) {};
				\draw[harc, bend left=10] (r1)edge(b1) (r2)edge(b2);
				\draw[harc, bend left=10] (b1)edge(r1) (b2)edge(r2);
				\draw[harc, bend left=10, Melon] (r1)edge(r2) (b1)edge(b2) (r2)edge(r1) (b2)edge(b1);
				\draw[harc, bend left=10, Melon] (r1)edge(b2) (r2)edge(b1) (b1)edge(r2)  (b2)edge(r1);
			\end{scope}
			
			\begin{scope}[shift={(11.1,0)}]
				\node[xsvertex, WildStrawberry] (r1) at (-1,-0.75) {};
				\node[xsvertex, WildStrawberry] (r2) at (-1,0.75) {};
				\node[xsvertex, MidnightBlue] (b1) at (1,-0.75) {};
				\node[xsvertex, MidnightBlue] (b2) at (1,0.75) {};
				\draw[harc, bend left=10] (b1)edge(r1) (b2)edge(r2);
				\draw[harc, bend left=10, Melon] (r1)edge(b1) (r2)edge(b2);
				\draw[harc, bend left=10, Melon] (r1)edge(r2) (b1)edge(b2) (r2)edge(r1) (b2)edge(b1);
				\draw[harc, bend left=10, Melon] (r1)edge(b2) (r2)edge(b1) (b1)edge(r2)  (b2)edge(r1);
			\end{scope}

			\begin{scope}[shift={(14.8,0)}]
				\node[xsvertex, WildStrawberry] (r1) at (-1,-0.75) {};
				\node[xsvertex, WildStrawberry] (r2) at (-1,0.75) {};
				\node[xsvertex, MidnightBlue] (b1) at (1,-0.75) {};
				\node[xsvertex, MidnightBlue] (b2) at (1,0.75) {};
				\draw[harc] (b1)edge(r1) (b2)edge(r2);
			\end{scope}
		\end{tikzpicture}
		
		\caption{Reflecting the arcs matching a red $E_2$ to a blue $E_2$ via bichromatic symmetrizations and inverse bichromatic symmetrizations.}
		\label{fig: symmetrizations}
	\end{figure}
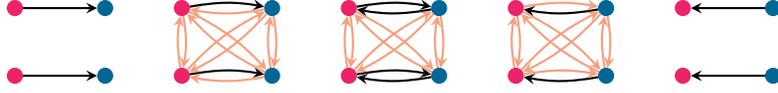

	\begin{lemma}
		If~$G$ and~$H$ are equivalent up to color changes
		and bichromatic
		symmetrization, then both graphs are ultrahomogeneous or neither of them is.
	\end{lemma}
	\begin{proof}
		Observe that the partial isomorphisms of $G$ correspond exactly to the partial isomorphisms of $H$ and, in particular, $\mathrm{Aut}(G)=\mathrm{Aut}(H)$. 
	\end{proof}

	\begin{lemma}\label{lemma: blocks-are-UH-subgraphs}
		Suppose $G$ is an ultrahomogeneous CCD and $\mathcal{B}$ is a block system of $\Aut(G)$. Then $G[X]$ is ultrahomogeneous for every block $X \in \mathcal{B}$. 
	\end{lemma}
	\begin{proof}
		Let $X$ be a block of an ultrahomogeneous CCD $G$ and let $\varphi:U\to U'$ be an isomorphism of induced subgraphs of $G[X]$.
		Since $G$ is ultrahomogeneous there exists $\varphi' \in \Aut(G)$ which extends $\varphi$.
		Since $X$ is a block and both sets~$U$ and~$U'$ are subsets of $X$ we obtain that $\varphi'(X) = X$.
		This implies that $\varphi'|_{X} \in \Aut(G[X])$ extends $\varphi$. 
	\end{proof}

	\begin{definition}\label{definition:blow:up}
		Let $R$ be a vertex color class of a CCD $G$, let~$H$ be a monochromatic CCD with a block system $\mathcal{B}$, and let~$\tau$ from~$\Aut(H)^\mathcal{B}$ to~$\Aut(G[R])$ be a permutational isomorphism. 
		Set
		\begin{align*}
			&\widehat{\tau}\colon (V(G)\setminus R) \dunion V(H)  \to V(G), \\
			&v \mapsto \begin{cases}
						\tau(B) &\text{if $v \in V(H)$ and $v \in B \in \mathcal{B}$}, \\
						v &\text{if } v \in V(G)\setminus R.
					  \end{cases}
		\end{align*}
		We define the \emph{blow-up~$G[H\rightarrow_{\tau} R]$ of~$R$ in~$G$ by~$H$ via~$\tau$} to be the CCD on vertex set~$(V(G)\setminus R)\dunion V(H)$ with colorings defined as follows. 
		\begin{enumerate}
			\item The vertex~$v$ has color~$\chi_G(\widehat{\tau}(v))$.
			\item The edge~$(v,v')$ has edge color
			${\ecol_H}{((v,v'))}$ if~$v,v'\in V(H)$ and edge color ${\ecol_G}{((\widehat{\tau}(v),(\widehat{\tau}(v'))))}$ otherwise.
		\end{enumerate}
		The blow-up is \emph{easygoing} if~$H$ is easygoing with respect to~$\mathcal{B}$.
	\end{definition}
	
	Abusing terminology, we sometimes also talk about a \emph{blow-up of~$R$ by~$X$} rather than a blow-up of~$R$ to~$H$ in case each graph induced by a block of~$H$ is isomorphic to~$X$. 
	Note, however, that our notion of blow-up is more general than similar notions of homogeneous blow-up defined~\cite{HeinrichSchneiderSchweitzer2020,truss:homogeneous-coloured-multipartite-graphs}. In particular the blocks of the graph~$H$ do not have to be homogeneously connected.

	\begin{lemma}\label{lem:blow:ups}
		Let $R$ be a color class of a CCD $G$ and let $H_1$ and $H_2$ be two ultrahomogeneous CCDs.
		If $Y_1\coloneqq  G[H_1\rightarrow_{\tau_1} R]$ and~$Y_2\coloneqq G[H_2\rightarrow_{\tau_2} R]$ are easygoing blow-ups, 
		then $Y_1$ is ultrahomogeneous if and only if $Y_2$ is ultrahomogeneous. 
		
		In particular, if~$G[R]$ is ultrahomogeneous, then~$Y_2$ is ultrahomogeneous if and only if~$G$ is ultrahomogeneous.
		
	\end{lemma}
	\begin{proof}
		Fix $i$ and $j$  in $\{1,2\}$.	
		For $M \subseteq V(Y_i)$ 
		set~$\overline{M}^j\coloneqq  \widehat{\tau}_j^{-1}(\widehat{\tau}_i(M))$.
		
		\medskip
		\noindent
		\textbf{Claim:} 
		Given 
		a partial isomorphism $\varphi\colon U\to W$ of~$Y_i$, there is a partial isomorphism~$\overline{\varphi}^j\colon \overline{U}^j\to \overline{W}^j$ of~$Y_{j}$ so that
		\[\varphi|_{U\cap (V(G)\setminus R)}=\overline{\varphi}^j|_{U\cap (V(G)\setminus R)}\] and
		 $\varphi$ and $\overline{\varphi}^j$ induce the same partial isomorphism of $G[R]$.
		
		\medskip
		\noindent
		$\ulcorner$ \textit{Proof of the claim.} 
		Let~$\mathcal{B}_i$ be the block system of~$H_i$ used to form the blow-up.
		
		Note that, because~$H_i$ is ultrahomogeneous, the map~$\varphi|_{V(H_i)}$ respects the blocks in~$\mathcal{B}_i$ in the sense that elements in the same block get mapped to elements in the same block. (Indeed there is an extension of~$\varphi|_{V(H_i)}$ to~$V(H_i)$ which respects the block system.) In particular, $\varphi|_{V(H_i)}$ induces a partial permutation of~$R$.
		Since~$\Aut(H_i)^{\mathcal{B}_i}$ and~$\Aut(H_j)^{\mathcal{B}_j}$ are permutationally isomorphic via~$\tau_i^{-1}\circ \tau_j$, there is an automorphism~$\widetilde{\varphi}^j$ of~$H_j$ mapping~$\overline{U}^j$ to~$\overline{W}^j$ so that~$\varphi$ and~$\widetilde{\varphi}^j$ induce the same permutation of~$R$.

		Consider the map~$\psi\colon \overline{U}^j \rightarrow \overline{W}^j$ for which~$\psi|_{\overline{U}^j\cap (V(G)\setminus R)}= \varphi$ and~$\psi|_{V(H_j)}= \widetilde{\varphi}^j$. 
		This map is a partial isomorphism by construction. It satisfies the conclusion of the claim.\hfill~$\lrcorner$

		\medskip
			It suffices now to assume that~$Y_1$ is ultrahomogeneous and prove that~$Y_2$ is also ultrahomogeneous.
			
		Let~$\varphi\colon U\rightarrow W$ be a partial isomorphism of~$Y_2$. Let~$\overline{\varphi}^1$ be the map given by the claim.
		Since we assume that~$Y_1$ is ultrahomogeneous $\overline{\varphi}^1$ extends to an automorphism~$\tau$ of~$Y_1$. Applying the claim again, we obtain a map~$\overline{\tau}^2 \in \Aut(Y_2)$ so that $\overline{\tau}^2$ and $\tau$, and thus also $\varphi$, induces the same map permutation of~$R$ and so that they agree on~$U\cap (V(G)\setminus R)$.
		
		Since~$G[H_2]$ is easygoing with respect to~$\mathcal{B}_2$ we can by Lemma~\ref{lem:application:of:easygoing} alter~$\tau$ without altering the induced action on~$\mathcal{B}_2$ but so that it agrees with~$\varphi$ on~$U\cap V(H_2)$
		such that the alteration is an automorphism since~$Y_2$ is a blow-up.
	\end{proof}

	If $H$ is the blow-up of the vertex color class $R$ of a monochromatic CCD $G$ (i.e., $G = G[R]$), then we also say $H$ is a blow-up of $G$.
	\begin{theorem}\label{thm:all:blow:ups:of:orient:and:unique}
		Let~$G$ and~$H$ be vertex-monochromatic ultrahomogeneous oriented graphs.
		If~$G$ is a blow-up of~$H$ and~$|G|>|H|>1$, then
		\[(G,H) \in \left\{(\cthree\cdot E_n, \cthree)\colon n \in \mathbb{N}_{\geq 2} \right\}, \left\{(E_n\cdot \cthree, E_n) \colon n \in \mathbb{N}_{\geq 2} \right\} \cup \left\{(\cfour, E_2)\right\}.\]
		In particular, a monochromatic ultrahomogeneous oriented graph arises in at most one way as a blow-up.
	\end{theorem}

	\begin{proof}	
	We analyze the list of ultrahomogeneous oriented monochromatic graphs (Theorem~\ref{thm: lachlan_asymmetric}) and obtain the list of combinations where one such graphs is a blow-up of another.
	By assumption $G = H[G\to_{\tau} V(H)]$ for some block system~$\mathcal{B}$ of $V(G)$ and  a permutational isomorphism
	$\tau \colon \Aut(G)^{\mathcal{B}}\to \Aut(H)$. 
	
	By assumption $|G| > |H|$ and, hence, $\mathcal{B}$ is a non-trivial block system of~$V(G)$.
	In particular, $G$ is not an edgeless graph.
	
	Four possibilities for~$G$ remain according to the list of Lachlan (Theorem~\ref{thm: lachlan_asymmetric}).
	
	If $G$ is a $\cfour$, then the only non-trivial block system $\mathcal{B}$ consists of the two diagonals of the $\cfour$ and $\Aut(G)^{\mathcal{B}} \cong \mathbb{Z}_2$. Since $\tau$ is a permutational isomorphism the only option for $H$ is $E_2$.
	
	In a similar fashion, we find that if $G$ is $\cthree \cdot E_n$ for some $n\in \mathbb{N}_{\geq 1}$, then $H = \cthree$, and, if $G$ is $E_n \cdot \cthree$ for some $n\in \mathbb{N}_{\geq 1}$, then $H = E_n$.
	
	The remaining case is that $G$ is $H_0$. The only non-trivial block system in this case consists of four isomophic copies of $E_2$ with $\Aut(G)^{\mathcal{B}} \cong SL(2,3)$. However, no graph with at most seven vertices on Lachlan's list has this automorphism group and, hence, no suitable permutational isomorphism $\tau$ exists.
	\end{proof}
	Note that the oriented graph~$\cfour$ as the blow-up of~$E_2$ is the only example among the oriented graphs where the connections between the blocks are not homogeneous.

		\section{The ultrahomogeneous oriented graphs with two vertex colors}\label{sec: bichromatic-classification}
		In this section we focus on oriented graphs~$G$ with two vertex color classes red~$R$ and blue~$B$.

		\paragraph*{Outline of the proof strategy.}
		Without loss of generality we assume~$|R|\geq |B|$.
		If $G$ is ultrahomogeneous, then~$G[R]$ is vertex-monochromatic and ultrahomogeneous. It is thus one of the graphs $\cfour$, $E_n$, $E_n\cdot \cthree$, $\cthree\cdot E_n$, or $H_0$ (Theorem~\ref{thm: lachlan_asymmetric} and Figure~\ref{fig: uh-monochr-graphs}).
		For each choice of~$G[R]$ we investigate how~$B$ can be connected to~$R$ making heavy use of the general extension theorem (Theorem~\ref{thm: general-extension}). 
		Specifically we set
		\begin{align*}
			N^R_+(b) &\coloneqq \{r \in R\colon (b,r) \in E(G)\},\\
			N^R_-(b) &\coloneqq \{r \in R\colon (r,b) \in E(G)\}, \text{ and}\\
			N^R(b) &\coloneqq N^R_+(b) \cup N^R_-(b).
		\end{align*}
		For each $b\in B$ it holds that $N^R_{-}(b) \cap N^R_{+}(b) = \emptyset$ and, hence, we set
		$A(b)$ to be the ordered partition obtained from $(N^R_+(b), N^R_-(b), R\setminus N^R(b))$ by deleting all empty parts.
		Conversely, for a partition~$A$ of $R$ we set $X(A) \coloneqq \{b \in B\colon A(b) = A\}$ and $\mathcal{B} \coloneqq \{X(A)\colon A \in \mathcal{A}\}$. 
		Since~$G$ is an oriented graph, we know that partitions~$A(b)$ have at most three parts, so~$|A(b)|\leq 3$ for every $b\in B$.
		
		We investigate partitions of~$G[R]$ with at most~$3$ parts and rule out the ones that are not ultrahomogeneous. For the ones that remain,~$B$ has least~$|\mathcal{A}(A(b))|$ elements and this will rule out most of the remaining cases (since $|B| \leq |R|$ by assumption).
		For the cases that nevertheless remain, we determine the permutation group induced by~$\Aut(G[R])$ on~$\mathcal{A}(A(b))$. This permutation group must be the group~$\Aut(G)^\mathcal{B}$ induced by~$\Aut(G[B])$ on one of its block systems.
		However,~$G[B]$ is ultrahomogeneous which limits the possible situations one final time. Table~\ref{tab:my-table} shows all the possible groups~$\Aut(G[B])^\mathcal{B}$ that can arise and the monochromatic graphs that admit them.
		What finally remains in the end are connections between~$R$ and~$B$ that indeed lead to ultrahomogeneous graphs.

		\begin{table}[t]
			\centering
			\caption{The oriented ultrahomogeneous digraphs $G$ according to Lachlan's classification (Theorem~\ref{thm: lachlan_asymmetric}).
				In the column \emph{blocks} we list for each possible non-trivial block system of $\Aut(G)$ 
				the isomorphism type of the subgraph induced by one (and thus by each) block.
				In the last column for each block system $\mathcal{B}$ we list the 
				induced action $\Aut(G)^{\mathcal{B}}$ on $\mathcal{B}$.}
		\label{tab:my-table}
		\begin{tabular}{|c|c|c|c|c|}
			\hline
			$G$ & $|V(G)|$ &$\Aut(G)$ & \begin{tabular}[c]{@{}c@{}}block induced graph\end{tabular} & $\Aut(G)^\mathcal{B}$ \\ \hline
			
			$E_n$ & $n$ & $\Sym(n)$ & \begin{tabular}[c]{@{}c@{}}$E_1$\end{tabular} & 
			\begin{tabular}[c]{@{}c@{}}$\Sym(n)$ \end{tabular} \\
			\hline
			
			$\cfour$ & $4$ & $\mathbb{Z}_4$ & \begin{tabular}[c]{@{}c@{}}$E_1$\\ $E_2$ \end{tabular} & \begin{tabular}[c]{@{}c@{}}$\mathbb{Z}_4$\\ $\mathbb{Z}_2$ \end{tabular} \\
			\hline
			
			$H_0$ & $8$ & $\mathrm{SL}(2,3)$ & \begin{tabular}[c]{@{}c@{}}$E_1$\\ $E_2$\end{tabular} & \begin{tabular}[c]{@{}c@{}}$\mathrm{SL}(2,3)$\\ $\Alt(4)$\end{tabular} \\ \hline
			
			$E_n \cdot \cthree$  & $3n$ & $\mathbb{Z}_3 \wr \Sym(n)$ & \begin{tabular}[c]{@{}c@{}}$E_1$\\ $\cthree$\end{tabular} & \begin{tabular}[c]{@{}c@{}}$\mathbb{Z}_3 \wr \Sym(n)$\\ $\Sym(n)$\end{tabular} \\ \hline
			
			$\cthree \cdot E_n$  & $3n$ & $\Sym(n) \wr \mathbb{Z}_3$ & \begin{tabular}[c]{@{}c@{}}$E_1$\\ $E_n$\end{tabular} & \begin{tabular}[c]{@{}c@{}}$\Sym(n) \wr \mathbb{Z}_3$\\ $\mathbb{Z}_3 $\end{tabular} \\ \hline
		\end{tabular}
	\end{table}

\begin{lemma}\label{lem: parts_C4}
Let $A$ be a partition of $V(\cfour)$ with~$|A|\leq 3$.
The set $\mathcal{A}(A)$ is a non-trivial ultrahomogeneous system of partitions with $|\mathcal{A}(A)|\leq 4$  if and only if~$A$ is a partition of~$\cfour$ into two independent sets of order 2.
The corresponding permutation group induced by~$\Aut(G)$ on~$\mathcal{A}(A)$ is $\mathbb{Z}_2$.
\end{lemma}

\begin{proof}
The ultrahomogeneous induced subgraphs of $\cfour$ are isomorphic to~$E_1$, to~$E_2$, or to~$\cfour$.
Let $u$ and $v$ be two non-adjacent vertices of $\cfour$ and let $\varphi$ be the transposition of $u$ and $v$.
Observe that there is a unique extension $\psi \in \Aut(\cfour)$ of $\varphi$ which also interchanges the other two vertices of $\cfour$.
In particular, there is no ultrahomogeneous partition of $\cfour$ which simultaneously contains 1-vertex parts and 2-vertex parts.
This settles the claim.
\end{proof}

\begin{corollary} \label{coro: c4}
Let $G$ be an ultrahomogeneous oriented graph on two vertex color classes~$R$ and~$B$ where $|R| \geq |B|$ and $G[R] \cong \cfour$.
If $R$ and $B$ are not homogeneously connected, then $G$ is equivalent (up to edge color changes) to one of the graphs of Figure~\ref{fig: R_c4}.
\end{corollary}
\begin{figure}[!h]
\centering
\begin{tikzpicture}[scale=.6]
	
	\begin{scope}[rotate=90]
		\node[xsvertex, WildStrawberry] (v1) at (0,-1.5) {};
		\node[xsvertex, WildStrawberry] (v2) at (0,-.5) {};
		\node[xsvertex, WildStrawberry] (v3) at (0,.5) {};
		\node[xsvertex, WildStrawberry] (v4) at (0,1.5) {};
		\node[xsvertex, MidnightBlue] (b1) at (1.2,-.5) {};
		\node[xsvertex, MidnightBlue] (b2) at (1.2,.5) {};
		\draw[harc] (v1) edge (v2) (v2) edge (v3) (v3) edge (v4);
		\draw[harc] (v4) edge[bend right=50] (v1);
		\draw[harc] (b1) edge (v1) (b1) edge (v3) (b2) edge (v2) (b2) edge (v4);
	\end{scope}
	
	\begin{scope}[shift ={(5,1)}]
		\node[xsvertex, WildStrawberry] (v1) at (0,.2) {};
		\node[xsvertex, WildStrawberry] (v2) at (0,-1) {};
		\node[xsvertex, WildStrawberry] (v3) at (1,.2) {};
		\node[xsvertex, WildStrawberry] (v4) at (1,-1) {};
		\node[xsvertex, MidnightBlue] (b1) at (3,.2) {};
		\node[xsvertex, MidnightBlue] (b2) at (3,-1) {};
		\node[xsvertex, MidnightBlue] (b3) at (4,.2) {};
		\node[xsvertex, MidnightBlue] (b4) at (4,-1) {};
		\draw[harc] (v1) edge (v2) (v2) edge (v3) (v3) edge (v4) (v4) edge (v1);
		\draw[harc] (b1) edge (b2) (b2) edge (b3) (b3) edge (b4) (b4) edge (b1);
		\draw[harc, bend right=35] (b1) edge (v1) (b1) edge (v3) (b3) edge (v1) (b3) edge (v3);
		\draw[harc, bend left=35] (b2) edge (v2) (b2) edge (v4) (b4) edge (v2) (b4) edge (v4);
	\end{scope}

\end{tikzpicture}
\caption{Up to color changes and bichromatic symmetrization, these are the ultrahomogeneous bichromatic oriented graphs containing~$\protect\cfour$ as a color class. 
}
\label{fig: R_c4}
\end{figure}
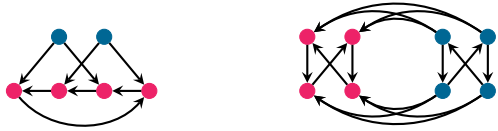

\begin{proof}
By Lemma~\ref{lem: parts_C4} we may assume that there is a vertex $b \in B$ such that
$A(b) = (\{r_1, r_3)\}, \{r_2, r_4\}))$.
Observe that $\mathbb{Z}_2$ acts on $\mathcal{A}(A(b))$.
According to Table~\ref{tab:my-table} the graphs $E_2$ and $\cfour$ are the only two oriented graphs with at most four vertices and a block system with suitable induced actions.
By Theorem~\ref{thm: general-extension} the result follows.
\end{proof}

\begin{lemma}\label{lem: parts-En}
If $A$ is an ordered partition of $V(E_n)$, then $\mathcal{A}(A)$ is an ultrahomogeneous system of partitions of $E_n$.
Moreover, if $\mathcal{A}(A)$ is non-trivial, $|\mathcal{A}(A)| \leq n$, and $2 \leq |A| \leq 3$, then
$A$ has precisely two parts one of which is a singleton.
The corresponding permutation group induced by~$\Aut(G)$ on~$\mathcal{A}(A)$ is $\Sym(n)$.
\end{lemma}

\begin{proof}
Observe that each two disjoint subsets of $V(E_n)$ are homogeneously connected and, hence, every vertex coloring of $E_n$ yields an ultrahomogeneous colored graph.
Denote the minimal cardinality of a part of $A$ by $p$.
If~$A$ has two parts and neither part is a singleton, then $n \geq p+2 \geq 2$ and, hence, $|\mathcal{A}(A)| \geq \binom{n}{2} \geq n$.
If $A$ has three parts of cardinality $p_1$, $p_2$, and $p_3$, respectively, then
$|\mathcal{A}(A)| = \binom{n}{p_1} \cdot \binom{n-p_1}{p_2}$ which is at least $n\cdot2$.
\end{proof}

\begin{corollary}\label{coro: en}
Let $G$ be an ultrahomogeneous oriented graph on two vertex color classes~$R$ and~$B$ where $|R| \geq |B|$ and $G[R] \cong E_n$ for some $n \in \mathbb{N}_{\geq 1}$.
If $R$ and $B$ are not homogeneously connected, then $G[B] \cong E_n$ and $R$ and $B$ are matching-connected.
\end{corollary}

\begin{proof}
Fix $b \in B$.
Since $G$ is oriented and $R$ and $B$ are not homogeneously connected we obtain $2 \leq |A(b)| \leq 3$.
By Theorem~\ref{thm: general-extension} the set $\mathcal{A}(A(b))$ is an ultrahomogeneous system of partitions.
Hence, we may apply Lemma~\ref{lem: parts-En}, which yields that without loss of generality $A(b) = (\{v\}, V(E_n)\setminus \{v\})$.
Observe that $\Aut(E_n)^{\mathcal{A}(A(b))} = \Sym(n)$.
By Theorem~\ref{thm: general-extension} we have that $\Aut(G[B])^{\mathcal{B}} = \Sym(n)$.
According to Table~\ref{tab:my-table} the only candidate for~$G[B]$ is~$E_n$ (since all other graphs which have a block system on which $\Sym(n)$ acts have more than $n$ vertices).
\end{proof}


\begin{lemma} \label{lem: parts-C3}
If $A$ is a partition of $\cthree\cdot E_n$, then $\mathcal{A}(A)$ is a non-trivial ultrahomogeneous system of partitions with $|\mathcal{A}(A)| \leq 3n$ if and only if $A = (P_1, P_2, P_3)$ such that $G[P_i] \cong E_n$ for each $i \in [3]$.
The corresponding permutation group induced by $\Aut(G)$ on $\mathcal{A}(A)$ is $\mathbb{Z}_3$.
\end{lemma}

\begin{proof}
Let $A$ be a non-trivial ultrahomogeneous partition of $G\coloneqq \cthree\cdot E_n$ and let~$P$ be a part of $A$.
An induced ultrahomogeneous subgraph of~$\cthree\cdot E_n$ is either isomorphic to $\cthree\cdot E_{n'}$ for some $n' \in [n-1]$ or to $E_{n''}$ for some $n'' \in [n]$.

If $G[P] \cong \cthree\cdot E_{n'}$ for some $n' \in [n-1]$, then there are $\binom{n}{n'}^3$ isomorphic copies of $G[P]$ in~$G$.
Since~$G$ is ultrahomogeneous every isomorphic copy of~$P$ belongs to at least one ultrahomogeneous partition of $G$.
Hence,
\[|\mathcal{A}(A)|	 \geq \binom{n}{n'}^3 > 3n = |V(G)|,\] which is a contradiction.

If $G[P] \cong E_{n''}$ for some $n'' \in [n]$, then there are $3\binom{n}{n''}$ isomorphic copies of~$G[P]$ induced in $G$, which contradicts $|\mathcal{A}(A)| \leq 3n$ unless $n = n''$. In this case~$A$ has three parts each of which induces an $E_n$.
\end{proof}

\begin{lemma}\label{lem: parts en-c3} \label{lem: parts en cthree}
If $A$ is a partition of $G \coloneqq E_n \cdot \cthree$ with~$|A|\leq 3$, then $\mathcal{A}(A)$ is a non-trivial ultrahomogeneous system of partitions if and only if each part $P$ of~$A$ satisfies $G[P] \cong E_{n'} \cdot \cthree$ for some $n' \in [n-1]$.
\end{lemma}
\begin{proof}
Let $A$ be a non-trivial partition of $E_n \cdot \cthree$ and let $P$ be a part of $A$.
Observe that either $G[P] \cong E_{n'} \cdot \cthree$ for some $n' \in [n-1]$ or $G[P] \cong E_{n''}$ for some $n'' \in [n]$.

If $G[P] \cong E_{n''}$ for some $n'' \in [n-1]$, then there are $\binom{n}{n''}3^{n''}$ isomorphic copies of $G[P]$ in $G$.
Observe that $\binom{n}{n''}3^{n''} > 3n$ unless $n=2$ and $n''=1$. In this case $|P| = 1$ and hence, $A$ has at least four parts (the triangle which contains the vertex of $P$ must be partitioned into three vertices by Lemma~\ref{lemma: blocks-are-UH-subgraphs}) and the other triangle forms at least one more part. This is a contradiction to~$|A|\leq 3$.

We conclude that no part of $A$ induces an edgeless graph, that is, every part of $A$ induces a wreath product of an edgeless graph with a~$\cthree$. This settles the claim.
\end{proof}

\begin{lemma} \label{lem: part-h0}
The only ultrahomogeneous partition of $H_0$ into at most three parts is the trivial partition.
\end{lemma}
\begin{proof}
Let $A$ be a non-trivial ultrahomogeneous partition of $H_0$ into at most three parts.
The induced ultrahomogeneous subgraphs of $H_0$ are $E_1$, $E_2$, $\cthree$, $\cfour$, and $H_0$.
Since $|V(H_0)| = 8$ and $A$ has at most three parts we obtain that at least one part $P$ of $A$ induces a $\cthree$ or a $\cfour$.
Let $P'$ be a part of $A$ distinct from $P$.
Since~$A$ is an ultrahomogeneous partition the graph $D \coloneqq H_0[P \cup P']$ with a vertex coloring such both parts $P$ and $P'$ are a color class of $D$ is ultrahomogeneous.
In particular, $D$ is among the graphs of Lemma~\ref{lem: parts_C4} and Lemma~\ref{lem: parts-C3}.
Only one of these graphs appears as induced subgraph in $H_0$, namely the graph which contains two parts $P$ and $P'$ that both induce $\cthree$ and $H_0[P\cup P']$ with the $\chi_G(A)$ coloring is isomorphic to the graph on the right in Figure~\ref{fig: blowup}. However, then the graph induced by $P$ and $V(H_0) \setminus (P\cup P')$ is not ultrahomogeneous, which contradicts the assumption.
Thus $A$ is trivial.
\end{proof}

\begin{corollary}\label{coro: h0}
If $G$ is an ultrahomogeneous oriented graph on two vertex color classes~$R$ and~$B$ such that $G[R] \cong H_0$,
then $R$ and $B$ are homogeneously connected.
\end{corollary}

\begin{proof}
This is an immediate consequence of Lemma~\ref{lem: part-h0} and Theorem~\ref{thm: general-extension}.
\end{proof}

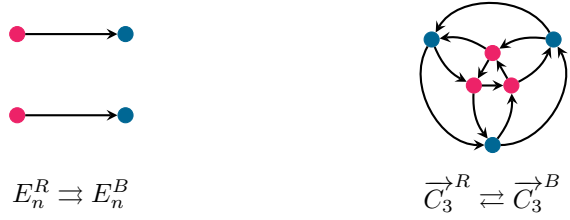
\begin{figure} 
\centering
\begin{tikzpicture}[scale=.8]
	\begin{scope}[scale=.9]
		\node[xsvertex, WildStrawberry] (r1) at (-1,-0.75) {};
		\node[xsvertex, WildStrawberry] (r2) at (-1,0.75) {};
		\node[xsvertex, MidnightBlue] (b1) at (1,-0.75) {};
		\node[xsvertex, MidnightBlue] (b2) at (1,0.75) {};
		\draw[harc] (r1)edge(b1) (r2)edge(b2);
		\node[] (match) at (0,-2.2) {$E_n^R\rightrightarrows E_n^B$};
	\end{scope}
	
	\begin{scope}[shift={(7,0)}, scale=.9]
		\def\irad{.4cm};
		\def\vxnumber{3}
		\def\angle{360/\vxnumber}
		\foreach \i in {1,...,\vxnumber}{
			\node[xsvertex, WildStrawberry] (i\i) at (270-\angle/2+\i*\angle:\irad) {};
		}
		\draw[harc] (i1) edge (i2) (i2) edge (i3) (i3) edge (i1);
		\def\orad{1.3cm};
		\foreach \i in {1,...,\vxnumber}{
			\node[xsvertex, MidnightBlue] (o\i) at (210-\angle/2+\i*\angle:\orad) {};
		}
		\draw[harc, bend right = 60] (o1) edge (o2) (o2) edge (o3) (o3) edge (o1);
		\draw[harc, bend right = 20] (i3) edge (o1) (o2) edge (i2);
		\draw[harc, bend right = 20] (i1) edge (o2) (o3) edge (i3);
		\draw[harc, bend right = 20] (i2) edge (o3) (o1) edge (i1);
		\node[] (cthree) at (0,-2.2) {$\cthree^R \rightleftarrows \cthree^B$};
		
	\end{scope}
\end{tikzpicture}
\caption{Up to equivalence all ultrahomogeneous oriented bichromatic graphs can be obtained from blow-ups of graphs of the above form.}\label{fig: blowup}
\end{figure}

Denote by~${E_n^R}\rightrightarrows{E^B_n}$ the bichromatic graph whose color classes induce a red and a blue~$E_n$ and with a perfect matching directed from red to blue (Figure~\ref{fig: blowup} left). Also denote by $\cthree^R \rightleftarrows \cthree^B$ the ultrahomogeneous oriented  graph depicted in Figure~\ref{fig: blowup} on the right non-trivially joining two directed triangles with a directed 6-Cycle~$\overrightarrow{C_6}$.

For blow-ups of bichromatic graphs~${E_n^R}\rightrightarrows{E^B_n}$ we introduce a specialized notation. We use the notation~${E_n^R}\rightrightarrows{E^B_n}\uparrow \mathcal{C}$ for a set of CCDs~$\mathcal{C}$ to indicate that one or possibly two of the sides can be blown up to one of the graphs in~$\mathcal{C}$. Similarly we use the notation~$\cthree^R \rightleftarrows \cthree^B\uparrow \mathcal{C}$.

\begin{theorem}\label{bichromatic:theorem}
Let $G$ be an oriented bichromatic graph. Then $G$ is ultrahomogeneous if and only if 
it is equivalent (up to color changes and (inverse) bichromatic symmetrizations) to one of following graphs (see Figures~\ref{fig: R_c4} and~\ref{fig: blowup}):

\begin{enumerate}
	\item $G_1 \mathbin{\square} G_2$ with~$G_1,G_2\in \{\cfour, H_0,E_n,E_n\cdot \cthree,\cthree\cdot E_n\}$ (disjoint union),
	\item ${E_n^R}\rightrightarrows {E^B_n}\uparrow \{E_n\cdot \cthree\colon n \in \mathbb{N}_{>1} \}$ (matching, possibly blown up to~$E_n \cdot \cthree$),
	\item ${E_2^R}\rightrightarrows {E^B_2}\uparrow \{E_2\!\cdot\! \cthree,\cfour\}$ (matching, possibly blown up to~$E_2\cdot \cthree$ or to~$\cfour$), or
	\item $\cthree^R \rightleftarrows \cthree^B\uparrow \{\cthree \cdot E_{n}\colon n\in \mathbb{N}_{>1}\}$ (special triangle connection, possibly blown up to~$\cthree \cdot E_n$).
\end{enumerate}
\end{theorem}
\begin{proof}
Let $G$ be an ultrahomogeneous oriented graph on precisely two vertex color classes $R$ and $B$.
We may assume without loss of generality that $|R| \geq |B|$.
By Theorem~\ref{thm: lachlan_asymmetric} each of the graphs $G[R]$ and $G[B]$ is isomorphic to one of the following graphs:
$\cfour$, $H_0$, $E_n$ for some $n \in \mathbb{N}_{\geq 1}$, $\cthree \cdot E_n$ for some $n \in \mathbb{N}_{\geq 1}$, or $E_n \cdot \cthree$ for some $n \in \mathbb{N}_{\geq 1}$.

If $G[R] \cong \cfour$, then by Corollary~\ref{coro: c4} either $R$ is homogeneously connected to~$B$ or $G$ is among the two graphs of Figure~\ref{fig: R_c4}. Observe that both of these graphs are blow-ups of the graph ${E_2^R}\rightleftarrows {E^B_2}$ (left in Figure~\ref{fig: blowup}).

If $G[R] \cong H_0$, then $R$ is homogeneously connected to $B$ by Corollary~\ref{coro: h0}.

If $G[R] \cong E_n$ for some $n \in \mathbb{N}_{\geq 1}$, then by Corollary~\ref{coro: en} either $R$ and $B$ are homogeneously connected or $G[B] \cong E_n$ and $R$ and $B$ are matching-connected.

If $G[R] \cong \cthree \cdot E_n$ for some $n \in \mathbb{N}_{\geq 1}$, then by Lemma~\ref{lem: parts-C3} either $R$ and $B$ are homogeneously connected or for each vertex $b \in B$ we have $A(b) = (P_1, P_2, P_3)$ with $P_i \cong E_n$ for $i \in [3]$.
Observe that $\mathbb{Z}_3$ acts on $\mathcal{A}(A(b))$.
From Theorem~\ref{thm: general-extension} and Table~\ref{tab:my-table} we obtain $G[B] \cong \cthree \cdot E_{n'}$ for some $n' \in [n]$. From the structure of $A(b)$ we obtain that $G$ is a blow-up of the graph on the right in Figure~\ref{fig: blowup}.

If $G[R] \cong E_n \cdot \cthree$ for some $n \in \mathbb{N}_{\geq 1}$, then by Lemma~\ref{lem: parts en-c3} either~$R$ and~$B$ are homogeneously connected or for each $b \in B$ every part $P$ of $A(b)$ is isomorphic to $E_{n'}\cdot\cthree$ for some $n' \in [n-1]$.
If $A(b)$ has three parts $P_1$, $P_2$, and $P_3$ such that $P_i \cong E_{n_i}\cdot \cthree$, then $|\mathcal{A}(A(b))| > 3n$ unless $n=3$ and $n_1 = n_2 = n_3 = 1$.
The corresponding permutation group admitted by the set $\mathcal{A}(A(b))$ is $\Sym(3)$.
According to Theorem~\ref{thm: general-extension} and Table~\ref{tab:my-table}, $G[B] \cong E_3$ or $G[B] \cong E_3 \cdot \cthree$. In the first case, $R\times B$ is a perfect matching of a red $E_3$ with a blue $E_3$, and the second case is a blow-up (up to edge color changes) of this one.
The last remaining case is that $A(b) = (P_1, P_2)$ with $P_i \cong E_{n_i}\cdot \cthree$ for $i \in [2]$. 
Set $n_i = |P_i|$ for $i \in [2]$.
If $\min_{i \in [2]}  n_i =1$, then $|\mathcal{A}(A(b))| = n$ with $\Sym(n)$ as the corresponding permutation group.
We obtain $G[B] \cong E_n$ or $G[B] \cong E_n\cdot \protect\cthree$. In the first case, $R\times B$ is a perfect matching of a red $E_n$ with a blue $E_n$, and the second case is (up to edge color changes) a blow-up of this one.
If $\min_{i \in [2]}n_i \geq 2$, then there are partitions $A_1, A_2, A_3$, and $A_4$ in $\mathcal{A}(A(b))$ such that $\pi_1(A_1) \cap \pi_1(A_2) = \emptyset$ and $\pi_1(A_3) \cap \pi_1(A_3) =1$.
In particular $\ecol_G(b_{A_2}, b_{A_1}) = \ecol_G(b_{A_1}, b_{A_2}) \neq \ecol_G(b_{A_3} b_{A_4}) )= \ecol_G(b_{A_4}, b_{A_3})$, which is a contradiction since oriented graphs only admit one symmetric relation (the non-edges).

To check that the listed graphs are ultrahomogeneous we can apply Theorems~\ref{thm:min-ext-thm} and~\ref{lem:blow:ups}.
\end{proof}

\section{The ultrahomogeneous vertex-colored o\-ri\-en\-ted graphs}\label{sec: final classif}

In this section we consider more than two vertex colors and finish the classification. Let us first analyze blow-ups.

\begin{lemma}\label{lem:blow:ups:are:consistent}
If a color class~$R$ in an ultrahomogeneous graph~$G$ induces a graph $G[R]$ that is a non-trivial blow-up, then~$G$ itself is a non-trivial blow-up to~$R$.
\end{lemma}
\begin{proof}
First observe that Theorem~\ref{bichromatic:theorem} shows the statement for the case of bichromatic graphs.
To conclude the proof of the lemma recall that Theorem~\ref{thm:all:blow:ups:of:orient:and:unique} shows that an oriented ultrahomogeneous graph arises in at most one way as a blow-up. (In particular the block systems of the blow-up agree regarding connections to all color classes.)
\end{proof}

Denote by~$E_n^{R_1}\rightrightarrows E_n^{R_2}\rightrightarrows \cdots \rightrightarrows E_n^{R_t}$ the
graph that consists of $k$ isomorphic copies of a transitive tournament on $t$ vertices with a discrete vertex coloring.
Denote by~$\cthree^{R_1} \rightleftarrows\cthree^{R_2} \rightleftarrows\cdots \rightleftarrows  \cthree^{R_t}$ the CCD graph with vertex colors~$R_1, R_2,\ldots,R_t$ for which each pair~$R_i,R_j$ induces~$\cthree \rightleftarrows\cthree$.
(Up to color changes
and (inverse) bichromatic symmetrizations
this graph is unique for each~$t$.)

Consistent with our previous notation, by~$G\uparrow \{H_1,\ldots\!,H_t\}$ we denote all graphs obtained from~$G$ by blowing up an arbitrary number of color classes of~$G$ to one of the graphs in~$\{H_1,\ldots,H_t\}$.

\begin{theorem}\label{thm: last thm}
If $G$ is an ultrahomogeneous vertex-colored oriented graph, then $G$ is equivalent (up to color changes
and (inverse) bichromatic symmetrization) to the color-disjoint union of the following graphs:
\begin{enumerate}
	\item $H_0$ (monochromatic),
	\item $E_n^{R_1}\rightrightarrows E_n^{R_2}\rightrightarrows \cdots \rightrightarrows E_n^{R_t}\uparrow \{E_n \cdot \cthree\colon n\in \mathbb{N}_{>1}\}$
	\item $E_2^{R_1}\rightrightarrows E_2^{R_2}\rightrightarrows \cdots \rightrightarrows E_2^{R_t}\uparrow \{E_2\!\cdot\! \cthree,\cfour\}$
	
	\item $\cthree^{R_1} \rightleftarrows\cthree^{R_2} \rightleftarrows\cdots  \rightleftarrows\cthree^{R_t}\uparrow \{\cthree \cdot E_{n}\colon n\in \mathbb{N}_{>1}\}$
\end{enumerate}
\end{theorem}

\begin{proof}
Let~$G$ be an ultrahomogeneous oriented graph.
We observe that by Lemma~\ref{lem:blow:ups:are:consistent} and Lemma~\ref{lem:blow:ups} it suffices to consider graphs
whose color classes induce ultrahomogeneous oriented graphs that are not the blow-up of another graph.

Thus the options, by Theorem~\ref{thm:all:blow:ups:of:orient:and:unique} are~$E_n$,~$C_3$, and~$H_0$.
By Theorem~\ref{bichromatic:theorem} classes inducing two different of these three graphs are homogeneously connected.

The graph~$H_0$ cannot be non-homogeneously connected to any other color class.

We argue that if~$R_1$,~$R_2$ and~$R_3$ are color classes such that~$G[R_i]\cong \cthree$, then any two of them induce $\cthree\rightleftarrows \cthree$. Assume that~$R_1$ and~$R_2$ as well as~$R_2$ and~$R_3$ and are non-homogeneously connected. It follows that, up to equivalence,~$G[R_1\cup R_2] \cong \cthree \rightleftarrows \cthree$ and~$G[R_2\cup R_3] \cong  \cthree \rightleftarrows \cthree$. It suffices to argue that~$R_1$ and~$R_3$ are not homogeneously connected. Suppose there are no edges between $R_1$ and $R_3$ in $G$. Observe that there are vertices~$r_1,r'_1\in R_1$ and~$r_3,r'_3\in R_3$ such that the distance of~$r_1$ and~$r_3$ is 2 while the distance of~$r'_1$ and~$r'_3$ is 3, which shows the statement.

We are left with the case in which all color classes induce independent sets and each pair of non-trivially color classes~$R$ and~$B$ is matching-connected. Such a graph is equivalent to an undirected ultrahomogeneous graph by forgetting the orientations (since there cannot be both a directed edge from~$R$ to~$B$ and a directed edge from~$B$ to~$R$ at the same time). The statement thus follows form the classification for undirected graphs~\cite{HeinrichSchneiderSchweitzer2020}.
\end{proof}

\section{Further research}
One of the most pressing questions is whether our approach can be used to classify the ultrahomogeneous countably infinite vertex-colored oriented graphs.
While certain concepts of our proof strategy are easily transferable to the countable infinite context (e.g., color classes induce monochromatic ultrahomogeneous structures, which are classified by Cherlin~\cite{MR1434988}) other proof concepts (e.g., most of our counting techniques) do not seem to transfer in a direct manner.

\paragraph*{Funding acknowledgement.} The research leading to these results has received funding from the European Research Council (ERC) under the European Union’s Horizon 2020 research and innovation programme (EngageS: grant agreement No. 820148).

	\bibliography{literature.bib}
	\bibliographystyle{alpha}
	
\end{document}